\documentclass[ejs,preprint]{imsart}

\usepackage{amsthm,amsmath, amsfonts, mathrsfs, hyperref, mathtools}
\usepackage{enumitem}
\RequirePackage[numbers]{natbib}
\RequirePackage[colorlinks,citecolor=blue,urlcolor=blue]{hyperref}

\allowdisplaybreaks

\doi{10.1214/154957804100000000}
\pubyear{0000}
\volume{0}
\firstpage{0}
\lastpage{0}

\startlocaldefs
\newtheorem{theorem}{Theorem}[]
\newtheorem{corollary}{Corollary}
\newtheorem{lemma}{Lemma}
\newtheorem{remark}{Remark}
\newtheorem{definition}{Definition}
\newtheorem{prop}{Proposition}
\DeclarePairedDelimiter\ceil{\lceil}{\rceil}
\DeclarePairedDelimiter\floor{\lfloor}{\rfloor}
\endlocaldefs

\begin{document}

\begin{frontmatter}

\title{Higher Order Langevin Monte Carlo Algorithm}
\runtitle{Higher Order Langevin Monte Carlo Algorithm}

\begin{aug}

\author{\fnms{Sotirios} \snm{Sabanis}\ead[label=e1]{s.sabanis@ed.ac.uk}}

\address{School of Mathematics, University of Edinburgh, UK\\
Alan Turing Institute, UK\\
\printead{e1}}

\author{\fnms{Ying} \snm{Zhang}\thanksref{t1}\ead[label=e2]{ying.zhang@ed.ac.uk}}

\address{Maxwell Institute for Mathematical Sciences, School of Mathematics, University of Edinburgh, UK\\
\printead{e2}}

\thankstext{t1}{Ying Zhang was supported by The Maxwell Institute Graduate School in Analysis and its Applications, a Centre for Doctoral Training funded by the UK Engineering and Physical Sciences Research Council (grant EP/L016508/01), the Scottish Funding Council, Heriot-Watt University and the University of Edinburgh. This work was supported by The Alan Turing Institute under the EPSRC grant EP/N510129/1 (Turing award number TU/B/000026).}

\runauthor{S. Sabanis and Y. Zhang}

\affiliation{University of Edinburgh and Alan Turing Institute}
\end{aug}

\begin{abstract}
A new (unadjusted) Langevin Monte Carlo (LMC) algorithm with improved rates in total variation and in Wasserstein distance is presented. All these are obtained in the context of sampling from a target distribution $\pi$ that has a density $\hat{\pi}$ on $\mathbb{R}^d$ known up to a normalizing constant. Moreover, $-\log \hat{\pi}$ is assumed to have a  locally Lipschitz gradient and its third derivative is locally H\"{o}lder continuous with exponent $\beta \in (0,1]$. Non-asymptotic bounds are obtained for the convergence to stationarity of the new sampling method with convergence rate $1+ \beta/2$ in Wasserstein distance, while it is shown that the rate is 1 in total variation even in the absence of convexity. Finally, in the case where $-\log \hat{\pi}$ is strongly convex and its gradient is Lipschitz continuous, explicit constants are provided.
\end{abstract}



 \received{\smonth{11} \syear{2018}}


\end{frontmatter}

\section{Introduction}
In Bayesian statistics and machine learning, one challenge, which has attracted substantial attention in recent years due to its high importance in data-driven applications, is the creation of algorithms which can efficiently sample from a high-dimensional target probability distribution $\pi$. In particular, its smooth version assumes that there exists a density on $\mathbb{R}^d$, denoted by $\hat{\pi}$, such that
\[
\hat{\pi} (x)= e^{-U(x)}/\int_{\mathbb{R}^d}e^{-U(y)}\,dy,
\]
with $\int_{\mathbb{R}^d}e^{-U(y)}\,dy< \infty$, where $U$ is typically continuously differentiable. Within such a setting, consider a filtered probability space $(\Omega, \mathscr{F}, (\mathscr{F}_{t})_{t \geq 0}, \mathbb{P})$, then the Langevin SDE associated with $\pi$ is defined by
\begin{equation}\label{eqn:sde}
dx_t = -\nabla U(x_t)dt +\sqrt{2}dw_t,
\end{equation}
where $(w_t)_{t \geq 0}$ is a $d$-dimensional Brownian motion. It is a classical result that under mild conditions, the SDE \eqref{eqn:sde} admits $\pi$ as its unique invariant measure. The corresponding numerical scheme of the Langevin SDE obtained by using the Euler-Maruyama (Milstein) method yields the unadjusted Langevin algorithm (ULA), known also as the Langevin Monte Carlo (LMC), which has been well studied in the literature. For a globally Lipschitz $\nabla U$, the non-asymptotic bounds in total variation and Wasserstein distance between the $n$-th iteration of the ULA algorithm and $\pi$ have been provided in  \cite{dal16}, \cite{DM16} and \cite{DM17}. As for the case of superlinear $\nabla U$, the difficulty arises from the fact that the algorithms constructed based on explicit numerical schemes, for example ULA, is unstable (see \cite{msh}), and its Metropolis adjusted version, MALA, loses some of its appealing properties as discussed in \cite{RT96} and demonstrated numerically in \cite{tula}.  However, recent research has developed new types of explicit numerical schemes for SDEs with superlinear coefficients,  and it has been shown in \cite{Kruse et al.}, \cite{hutzenthaler2012}, \cite{milsteinscheme}, \cite{eulerscheme}, \cite{SabanisAoAP}, \cite{WangGan}, that the tamed Euler (Milstein) scheme converges to the true solution of the SDE \eqref{eqn:sde} in $\mathscr{L}^p$ on any given finite time horizon with optimal rate. This progress led to the creation of the tamed unadjusted Langevin algorithm (TULA) in \cite{tula}, where the aforementioned convergence results are extended to an infinite time horizon and, moreover, one obtains rate of convergence results in total variation and in Wasserstein distance.

The new higher order LMC algorithm (HOLA) considered in this article has the following representation, for any $n \in \mathbb{N}$,
\begin{equation}\label{eqnscheme2}
\overline{X}_{n+1} = \overline{X}_n+\mu_{\gamma}(\overline{X}_n)\gamma+\sqrt{2\gamma}\sigma_{\gamma}(\overline{X}_n)Z_{n+1},
\end{equation}
where $\gamma \in (0,1)$ is the step size, $(Z_n)_{n \in \mathbb{N}}$ are i.i.d. standard $d$-dimensional Gaussian random variables, for all $x \in \mathbb{R}^d$, 
\[
\mu_{\gamma}(x) = -\nabla U_{\gamma}( x) +(\gamma/2) (\left(\nabla^2 U\nabla U\right)_{\gamma}( x)-\vec{\Delta}(\nabla U)_{\gamma}(x) ),
\] 
and 
\[
\sigma_{\gamma}(x) =\sqrt{\mathbf{I}_d - \gamma \nabla^2 U_{\gamma}(x)+(\gamma^2/3)(\nabla^2 U_{\gamma}(x))^2}\]
with $\mathbf{I}_d$ being the $d \times d$ identity matrix. The dependences of the coefficients on $\gamma$ are given by, for $x \in \mathbb{R}^d$
\begin{align} \label{deftamedcoeff}
\begin{split}
\nabla U_{\gamma}(x) 	&= \frac{\nabla U(x)}{(1+\gamma^{3/2}|\nabla U(x)|^{3/2})^{2/3}}, \\
\nabla^2 U_{\gamma}(x)	&= \frac{\nabla^2 U(x)}{1+\gamma|\nabla^2 U(x)|}, \\
\left(\nabla^2U\nabla U \right)_{\gamma}(x) 	&= \frac{\nabla^2U(x)\nabla U(x)}{1+\gamma|x||\nabla^2U(x)||\nabla U(x)|},\\
\vec{\Delta}(\nabla U)_{\gamma}(x)			&= \frac{\vec{\Delta}(\nabla U)(x)}{1+\gamma^{1/2}|x||\vec{\Delta}(\nabla U)(x)|}.
\end{split}
\end{align}
The tamed coefficients in \eqref{deftamedcoeff} are chosen such that the exponential moments and the desired rate of convergence of the scheme can be obtained, see Section \ref{relatedwork} for further discussions. One notes that $\sigma_{\gamma}^2$ is a positive definite matrix which has a unique square-root. In practice, $\sigma_{\gamma}$ can be computed by generating two independent standard Gaussian noise $\xi$ and $\tilde{\xi}$, then one considers $\left(\mathbf{I}_d- (1/2) \gamma \nabla^2 U_{\gamma}(X_n)\right)\xi_{n+1}+(\sqrt{3}/6)\gamma \nabla^2 U_{\gamma}(x)\tilde{\xi}_{n+1}$, which has the same distribution as $\sigma_{\gamma}(X_n)Z_{n+1}$ (see \cite{dal17user} and Chapter 10.4 in \cite{kloeden2011numerical}). The HOLA algorithm \eqref{eqnscheme2} is constructed based on the order 1.5 scheme \eqref{eqnscheme} of the SDE \eqref{eqn:sde}, which is obtained using the It\^o-Taylor (known also as Wagner-Platen) expansion (see Chapter 10 in \cite{kloeden2011numerical}) and can be written explicitly as:
\begin{align} \label{eqnscheme}
\begin{split}
X_{n+1}& = X_n -\nabla U_{\gamma}( X_n)\gamma +\frac{\gamma^2}{2}\left(\left(\nabla^2 U\nabla U\right)_{\gamma}( X_n)-\vec{\Delta}(\nabla U)_{\gamma}(X_n)\right)\\
		&\hspace{9em} + \sqrt{2\gamma}\bar{Z}_{n+1}-\sqrt{2}\nabla^2 U_{\gamma}(X_n)\tilde{Z}_{n+1}
\end{split}
\end{align}
where $(\bar{Z}_n)_{n \in \mathbb{N}}$ are i.i.d. standard $d$-dimensional Gaussian random variables,  and $(\tilde{Z}_n)_{n \in \mathbb{N}}$ are i.i.d. $d$-dimensional Gaussian random variables with mean $\mathbf{0}$ and covariance $\frac{1}{3}\gamma^3\mathbf{I}_d$ defined by $\tilde{Z}_{n+1} = \int_{t_n}^{t_{n+1}}\int_{t_n}^s\,dw_r\,ds$. Moreover, $\bar{Z}_{n+1}$ and $\tilde{Z}_{n+1}$ are not independent, and for any $n \in \mathbb{N}$, $k, l = 1, \dots, d$,
\begin{align*}
\mathbb{E}\left(\sqrt{\gamma}\bar{Z}_{n+1}^{(k)}  \tilde{Z}_{n+1}^{(l)} \right)=
\begin{cases}
\frac{1}{2}\gamma^2, 	& \quad \mathrm{for} \, k = l,\\
0,							& \quad \mathrm{otherwise.}
\end{cases}
\end{align*} 
One observes that the scheme $\eqref{eqnscheme}$ is Markovian, and $Law(X_n)$ is the same as $Law(\overline{X}_n)$, for any $n \in \mathbb{N}$.

For the HOLA algorithm \eqref{eqnscheme2}, by extending the techniques used in \cite{tula} and  \cite{SabanisYing}, it can be shown that the scheme \eqref{eqnscheme2} has a unique invariant measure $\pi_{\gamma}$, and one can obtain convergence results between $\pi_{\gamma}$ and the target distribution $\pi$ in some proper distance. More precisely, assume the potential $U$ is three times differentiable, and its third derivative is locally H\"{o}lder continuous with exponent $\beta \in (0,1]$. Then, under certain conditions (specified in Section \ref{mainresults}), Theorem \ref{thmwasserstein} and \ref{thmtv} state that the rate of convergence between the $n$-th iteration of the new algorithm and the target measure $\pi$ is $1+\beta/2$ in Wasserstein distance, whereas the rate is 1 in total variation. Here, one notes that these results are obtained in the context of having superlinear $\nabla U$. To the best of the authors' knowledge, these are the first such results which provide a higher rate of convergence in Wasserstein distance compared to the existing literature. As for the total variation distance, \cite{DM16} proves that the rate of convergence is 1 for the case of a strongly convex $U$, whereas our result yields the same convergence rate without assuming convexity.

The paper is organised as follows. Section \ref{mainresults} presents the assumptions and main results in both super-linear and Lipschitz settings. Section \ref{relatedwork} discusses the contribution of our work with comparison to the existing literature. In Section \ref{locallipcase}, the proofs of Theorem \ref{thmwasserstein} and Theorem \ref{thmtv} are provided, while the proofs of Theorem \ref{thmwassersteinlip} and Corollary \ref{thmwassersteingau} can be found in Section \ref{lipcase}. An example is provided in Section \ref{exlip} illustrating the applicability of the proposed algorithm in the Lipschitz case. Auxiliary results are provided in Appendices.

We conclude this section by introducing some notation. The Euclidean norm of a vector \(b \in \mathbb{R}^d\), the spectral norm and the Frobenius norm of a matrix \(\sigma \in \mathbb{R}^{d\times m}\) are denoted by $|b|$, $|\sigma|$  and  $|\sigma|_{\mathsf{F}}$ respectively. $\sigma^{\mathsf{T}}$ is the transpose matrix of $\sigma$. The $i$-th element of $b$ and \((i,j)\)-th element of $\sigma$ are denoted respectively by $b^{(i)}$ and \(\sigma^{(i,j)}\), for every \(i = 1, \dots,d\) and \(j = 1,\dots,d\). In addition, denote by \(\floor{a}\) the integer part of a positive real number $a$, and $\ceil{a} = \floor{a}+1$. The inner product of two vectors \(x,y \in \mathbb{R}^d\) is denoted by $xy$. For all $x \in \mathbb{R}^d$ and $M>0$, denote by $\mathrm{B}(x,M)$ (respectively $\overline{\mathrm{B}}(x,M)$) the open (respectively close) ball centered at $x$ with radius $M$. Let $f:\mathbb{R}^{d} \rightarrow \mathbb{R}$ be a twice continuously differentiable function. Denote by $\nabla f$, $\nabla^2 f$ and $\Delta f$ the gradient of $f$, the Hessian of $f$ and the Laplacian of $f$ respectively. Denote by $\vec{\Delta}g: \mathbb{R}^d \to \mathbb{R}^d$ the vector Laplacian of $g$, i.e., for all $x \in \mathbb{R}^d$,  $\vec{\Delta}g(x)$ is a vector in $\mathbb{R}^d$ whose $i$-th entry is $\sum_{u=1}^d \frac{\partial^2 g^{(i)}}{\partial x^{(u)}\partial x^{(u)}}(x)$. For $m, m' \in \mathbb{N}^{\ast}$, define
\begin{align*}
&C_{\mathrm{poly}}(\mathbb{R}^m, \mathbb{R}^{m'}) \\
&\hspace{3em}= \left\{P \in C(\mathbb{R}^m, \mathbb{R}^{m'}) | \exists C_q, q \geq 0, \forall x \in \mathbb{R}^d, \,  |P(x)|\leq C_q(1+|x|^q)\right\}.
\end{align*}
For any $t \geq 0$, denote by $\mathcal{C}([0,t], \mathbb{R}^d)$ the space of continuous $\mathbb{R}^d$-valued paths defined on the time interval $[0,t]$. 

Let $\mu$ be a probability measure on $(\mathbb{R}^d, \mathcal{B}(\mathbb{R}^d))$ and $f$ be a $\mu$-integrable function, define $\mu(f) = \int_{\mathbb{R}^d}f(x)\,d\mu(x)$. Given a Markov kernel $R$ on $\mathbb{R}^d$ and a function $f$ integrable under $R(x, \cdot)$, denote by $Rf(x) = \int_{\mathbb{R}^d}f(y)R(x,dy)$. Let $V: \mathbb{R}^d \rightarrow [1,\infty)$ be a measurable function. The $V$-total variation distance between $\mu$ and $\nu$ is defined as $\|\mu-\nu\|_V = \sup_{|f|\leq V}|\mu(f)-\nu(f)|$. If $V=1$, then $\|\cdot\|_V$ is the total variation denoted by $\|\cdot\|_{TV}$. Let  $\mu$ and $\nu$ be two probability measures on a state space $\Omega$ with a given $\sigma$-algebra. If $\mu \ll \nu$, we denote by $d\mu/d\nu$ the Radon-Nikodym derivative of $\mu$ w.r.t. $\nu$. Then, the Kullback-Leibler divergence of $\mu$ w.r.t. $\nu$ is given by
\[
\mathrm{KL}(\mu|\nu) = \int_{\Omega}\frac{d\mu}{d\nu}\log \left(\frac{d\mu}{d\nu}\right)\, d\nu.
\]

We say that $\zeta$ is a transference plan of $\mu$ and $\nu$ if it is a probability measure on $(\mathbb{R}^d \times \mathbb{R}^d, \mathcal{B}(\mathbb{R}^d)\times \mathcal{B}(\mathbb{R}^d))$ such that for any Borel set $A$ of $\mathbb{R}^d$, $\zeta(A \times \mathbb{R}^d)$= $\mu (A)$ and $\zeta(\mathbb{R}^d \times A) = \nu (A)$. We denote by $\Pi(\mu, \nu)$ the set of transference plans of $\mu$ and $\nu$. Furthermore, we say that a couple of $\mathbb{R}^d$-valued random variables $(X,Y)$ is a coupling of $\mu$ and $\nu$ if there exists $\zeta \in \Pi(\mu, \nu)$ such that $(X, Y)$ is distributed according to $\zeta$. For two probability measures $\mu$ and $\nu$, the Wasserstein distance of order $p \geq 1$ is defined as
\[
W_p(\mu, \nu) = \left(\inf_{\zeta \in \Pi (\mu, \nu)} \int_{\mathbb{R}^d \times \mathbb{R}^d}|x-y|^p\, d\zeta(x, y) \right)^{1/p}.
\]

\section{Main results}\label{mainresults}
Assume $U$ is three times continuously differentiable. The following conditions are stated:
\begin{enumerate}[label=\textbf{H\arabic*}]
\item \label{h1} $\liminf_{|x|\rightarrow +\infty} |\nabla U(x)|=+\infty$, and $\liminf_{|x|\rightarrow +\infty}\frac{x\nabla U(x)}{|x||\nabla U(x)|}>0$.
\item \label{h3} There exists $L>0$, $\rho \geq 2$, and $\beta \in (0,1]$, such that for any $i = 1, \dots, d$ and for all $x,y \in \mathbb{R}^d$,
\[
|\nabla^2(\nabla U)^{(i)}(x)-\nabla^2(\nabla U)^{(i)}(y)|\leq L(1+|x|+|y|)^{\rho -2}|x-y|^{\beta},
\]
where $(\nabla U)^{(i)}$ denotes the $i$-th element of $\nabla U$.
\item \label{h4} $U$ is strongly convex, i.e. there exists $m>0$ such that for all $x, y \in \mathbb{R}^d$,
\[
(x-y)\left(\nabla U(x)-\nabla U(y)\right) \geq m |x-y|^2.
\]
\end{enumerate}
\begin{remark}
Unless otherwise specified, the constants $C, K>0$ may take different values at different places, but these are always independent of the step size $\gamma \in (0,1)$.
\end{remark}
\begin{remark} \label{poly} Assume \ref{h3} holds, then for any $i = 1, \dots, d$ and for all $x \in \mathbb{R}^d$,
\[
|\nabla^2(\nabla U)^{(i)}(x)|\leq K(1+|x|)^{\rho -2+\beta},
\]
where $K = \max\{L, |\nabla^2(\nabla U)^{(i)}(0)|\}$, moreover, for all $x, y \in \mathbb{R}^d$,
\[
|\nabla(\nabla U)^{(i)}(x)-\nabla(\nabla U)^{(i)}(y)|\leq K(1+|x|+|y|)^{\rho -2+\beta}|x-y|,
\]
which implies,
\[
|\nabla(\nabla U)^{(i)}(x)|\leq K_1(1+|x|)^{\rho-1+\beta },
\]
where $K_1 = \max \{K, |\nabla (\nabla U)^{(i)}(0)|\}$. Furthermore,
for all $x, y \in \mathbb{R}^d$,
\[
|\nabla U^{(i)}(x)-\nabla U^{(i)}(y)|\leq K_1(1+|x|+|y|)^{\rho-1+\beta }|x-y|,
\]
and one obtains
\[
|\nabla U^{(i)}(x)|\leq K_2(1+|x|)^{\rho+\beta },
\]
where $K_2 =  \max \{K_1, |\nabla U^{(i)}(0)|\}$. One notes that the above inequality implies
\[
|\vec{\Delta}(\nabla U)(x) -\vec{\Delta}(\nabla U)(y)| \leq d^{3/2}L(1+|x|+|y|)^{\rho-2}|x-y|^{\beta},
\]
\[
|\vec{\Delta}(\nabla U)(x)| \leq dK(1+2|x|)^{\rho -1}.
\]
\end{remark}
\begin{proof} See Appendix \ref{remark2}
\end{proof}

\begin{remark}\label{tametermbd}
By the definition of the tamed coefficients \eqref{deftamedcoeff} and \ref{h3}, one obtains for all $x \in \mathbb{R}^d$,
\[
|\nabla U_{\gamma}(x)| \leq \sqrt[3]{2}\gamma^{-1}, \quad |\nabla^2 U_{\gamma}(x)| \leq \gamma^{-1}, 
\]
\[
|\left(\nabla^2U \nabla U\right)_{\gamma}(x)| \leq (1+2^{2\rho +1}dK_1K_2)\gamma^{-1},
\]
\[
 |\vec{\Delta}(\nabla U)_{\gamma}(x)|\leq  (1+3^{\rho-1}dK)\gamma^{-1/2}.
\]
In particular, when $|x| \geq 1$, $x \in \mathbb{R}^d$, one obtains
 \[
|\nabla U_{\gamma}(x)| \leq \sqrt[3]{2}\gamma^{-1}, \quad |\nabla^2 U_{\gamma}(x)| \leq \gamma^{-1}, 
\]
\[
|\left(\nabla^2U\nabla U\right)_{\gamma}(x)| \leq\gamma^{-1}, \quad 
 |\vec{\Delta}(\nabla U)_{\gamma}(x)|\leq \gamma^{-1/2}.
\]
\end{remark}
The Markov kernel $R_{\gamma}$ associated with \eqref{eqnscheme2} is given by
\[
R_{\gamma}(x, A) = (2\pi)^{-d/2}\int_{\mathbb{R}^d} \mathbf{1}_A\left(x+ \mu_{\gamma}(x)\gamma+\sqrt{2\gamma}\sigma_{\gamma}(x)z\right)e^{-|z|^2/2} \,dz,
\]
for all $x \in \mathbb{R}^d$ and $A \in \mathcal{B}(\mathbb{R}^d)$. Denote by $(P_t)_{t \geq 0}$ the semigroup associated with \eqref{eqn:sde}. For all $x \in \mathbb{R}^d$ and $A \in \mathcal{B}(\mathbb{R}^d)$, we have $P_t(x, A)= \mathbb{E}[\mathbf{1}_{\{x_t \in A\}}|x_0 =x]$. In addition, for all $x \in \mathbb{R}^d$ and $h \in C^2(\mathbb{R}^d)$, the infinitesimal generator $\mathscr{A}$ associated with \eqref{eqn:sde} is defined by
\[
\mathscr{A}h(x) = -\nabla U(x)\nabla h(x)+\Delta h(x).
\]
For any $a>0$, define the Lyapunov function $V_a:\mathbb{R}^d \rightarrow [1,+\infty)$ for all $x \in \mathbb{R}^d$ by 
\[
V_a(x) = \exp\left(a(1+|x|^2)^{1/2}\right).
\]
Then, for the local Lipschitz drift, one obtains the following convergence results.
\begin{theorem}\label{thmwasserstein}
Assume \ref{h1}, \ref{h3} and \ref{h4} are satisfied. Then, there exist constants $C>0$ and $\lambda \in (0,1)$ such that for all $x \in \mathbb{R}^d$, $\gamma \in (0,1)$ and $n \in \mathbb{N}$,
\begin{equation}\label{wasserstein1}
W_2^2(\delta_x R_{\gamma}^n, \pi) \leq C (\lambda^{n\gamma}V_c(x)+\gamma^{2+\beta}),
\end{equation}
where $c$ is given in \eqref{constantc} and for all $\gamma \in (0,1)$,
\[
W_2^2(\pi_{\gamma}, \pi) \leq C\gamma^{2+\beta}.
\]
\end{theorem}
\begin{theorem}\label{thmtv}
Assume \ref{h1} and \ref{h3} are satisfied. There exist $C>0$ and $\lambda \in (0,1)$ such that for all $x \in \mathbb{R}^d$, $\gamma \in (0,1)$ and $n \in \mathbb{N}$,
\begin{equation}\label{tvthm1}
\|\delta_x R_{\gamma}^n-\pi\|_{V_c^{1/2}} \leq C (\lambda^{n\gamma}V_c(x)+\gamma),
\end{equation}
where $c$ is given in \eqref{constantc} and for all $\gamma \in (0,1)$,
\begin{equation}\label{tvthm2}
\|\pi_{\gamma}-\pi\|_{V_c^{1/2}} \leq C\gamma .
\end{equation}
\end{theorem}
In the case of super-linear coefficients, tracking the explicit constants involves tedious calculations, and it is less informative compared to the case of Lipschitz coefficients, in the sense that the dependence on the dimension of the constant $C$ (appearing in Theorem \ref{thmwasserstein} and Theorem \ref{thmtv}) is $O(e^d)$. One notes that this is due to the fact that exponential moments of the scheme \ref{eqnscheme2} is obtained when a log-Sobolev inequality is used. To illustrate the explicit dependence on the dimension, and to provide explicit constants for the moment bounds and the convergence in Wasserstein distance, the Lipschitz case is considered. Four times continuous differentiablility on $U$ is required and the following conditions are assumed:
\begin{enumerate}[label=\textbf{H\arabic*}]
\setcounter{enumi}{3}
\item \label{h5} There exists $L_1>0$, such that for all $x,y \in \mathbb{R}^d$, 
\[|\nabla U(x)-\nabla U(y)|\leq L_1|x-y|.\]
\item \label{h6} There exists $L_2>0$, such that for all $x,y \in \mathbb{R}^d$, 
\[|\nabla^2U(x)-\nabla^2U(y)|\leq L_2|x-y|.\]
\item \label{h7} There exists $L>0$, such that for all $x,y \in \mathbb{R}^d$, 
\[|\nabla^2(\nabla U)^{(i)}(x)-\nabla^2(\nabla U)^{(i)}(y)|\leq L|x-y|.\]
\end{enumerate}

One notices that, in the Lipschitz case, there is no need to use the tamed coefficients, and one can consult Theorem 10.6.3 in \cite{kloeden2011numerical} for the classical strong convergence result for the order 1.5 scheme in a finite time. The counterpart of algorithm \eqref{eqnscheme2} in the Lipschitz case becomes: for any $n \in \mathbb{N}$
\begin{equation}\label{eqnschemelip}
\tilde{X}_{n+1} = \tilde{X}_n+\mu (\tilde{X}_n)\gamma+\sqrt{2\gamma}\sigma(\tilde{X}_n)Z_{n+1},
\end{equation}
where $\gamma \in (0,1)$ is the step size, $(Z_n)_{n \in \mathbb{N}}$ are i.i.d. standard $d$-dimensional Gaussian random variables, for all $x \in \mathbb{R}^d$, $\mu (x) = -\nabla U ( x) +(\gamma/2) (\nabla^2 U(x)\nabla U ( x)-\vec{\Delta}(\nabla U) (x) )$, and $\sigma (x) =\sqrt{\mathbf{I}_d - \gamma \nabla^2 U(x)+(\gamma^2/3)(\nabla^2 U(x))^2}$. The Markov kernel $\tilde{R}_{\gamma}$ associated with \eqref{eqnschemelip} is given by
\[
\tilde{R}_{\gamma}(x, A) = (2\pi)^{-d/2}\int_{\mathbb{R}^d} \mathbf{1}_A\left(x+ \mu(x)\gamma+\sqrt{2\gamma}\sigma (x)z\right)e^{-|z|^2/2} \,dz,
\]
for all $x \in \mathbb{R}^d$ and $A \in \mathcal{B}(\mathbb{R}^d)$

\begin{theorem}\label{thmwassersteinlip}
Assume \ref{h4} - \ref{h7} are satisfied. Let $\gamma \in \left(0,\frac{1}{\tilde{m}} \wedge  \frac{8\tilde{m}^2}{m( 2L_1^2+7\tilde{m}L_1)}\right)$. Then, for all $x \in \mathbb{R}^d$ and $n \in \mathbb{N}$,
\[
W_2^2(\delta_x \tilde{R}_{\gamma}^n, \pi) \leq e^{-mn\gamma}\left(2|x-x^{\ast}|^2+\frac{2d}{m}\right)+\bar{C}\gamma^3,
\]
where $\tilde{m}$ is given in \eqref{eq:definition-tilde-m}, $\bar{C}=O(d^4)$ and its the explicit expression is given in the proof.
\end{theorem}
\begin{corollary}\label{thmwassersteingau}
Assume \ref{h4} - \ref{h7} are satisfied. Let $\gamma \in \left(0,\frac{1}{\tilde{m}} \wedge  \frac{8\tilde{m}^2}{m( 2L_1^2+7\tilde{m}L_1)}\right)$. If one considers a multivariate Gaussian as the target distribtuion, then for all $x \in \mathbb{R}^d$ and $n \in \mathbb{N}$,
\[
W_2^2(\delta_x \tilde{R}_{\gamma}^n, \pi) \leq e^{-mn\gamma}\left(2|x-x^{\ast}|^2+\frac{2d}{m}\right)+\tilde{C}\gamma^3.
\]
where $\tilde{m}$ is given in \eqref{eq:definition-tilde-m}, $\tilde{C}=O(d)$ and its the explicit expression is given in the proof.
\end{corollary}
\begin{remark} One notices that only three times continuous differentiablility on the potential $U$ is required in the case of super-linear coefficients, while we assume four times continuous differentiablility in the case of Lipschitz coefficients. This further smoothness in the Lipschitz case is reuqired in order to obtain a better dependence on the dimension of the bound in Wasserstein distance, i.e. to obtain $\bar{C} = O(d^4)$ in Theorem \ref{thmwassersteinlip}. While one can still obtain similar results in Theorem \ref{thmwassersteinlip} and Corollary \ref{thmwassersteingau} without assuming further smoothness, the dependence on dimension of the bound will increase to $O(d^6)$.
\end{remark}

\section{Related work and discussion} \label{relatedwork}
{\bf Higher order scheme.} The higher order LMC algorithm \eqref{eqnscheme2} is obtained using the It\^o-Taylor (Wagner-Platen) expansion, see \cite{Platen-Wagner} and Section 10.4 in \cite{kloeden2011numerical}. It is suggested in Section 10.6 in \cite{kloeden2011numerical}  that any higher order schemes can be constructed using such an approach. One notices that the LMCO' algorithm considered in \cite{dal17user}, which is obtained using the LMC algorithm with the Ozaki discretization, is close to the algorithm \eqref{eqnschemelip}, which is the counterpart of the algorithm \eqref{eqnscheme2} in the Lipschitz case. The difference between the two algorithms is that there is one more term $\vec{\Delta}(\nabla U)$ in \eqref{eqnschemelip}. Without this term, the rate of convergence of the algorithm \eqref{eqnschemelip} reduces from 1.5 to 1 in the Wasserstein-2 distance. \\
{\bf Tamed coefficients.} The algorithm \eqref{eqnscheme2} of the SDE \eqref{eqn:sde} with superlinear coefficient is constructed using a taming technique, which is first introduced in \cite{hutzenthaler2012} for the Euler scheme and is further developed in \cite{eulerscheme}. Then, a uniform taming approach is suggested in \cite{milsteinscheme} which allows natural extensions of the taming technique to higher order schemes. In other words, it suggests that each coefficient in the numerical scheme should be multiplied by the same taming factor (see Remark 2 in \cite{milsteinscheme}). However, in this article, different terms in the scheme \eqref{eqnscheme2} have different taming facotrs as defined in \eqref{deftamedcoeff}. The reason is that, instead of a direct application of It\^o's formula, one uses the derivation of the log-sobolev inequality to obtain exponential moment bounds for the numerical scheme  \eqref{eqnscheme2} in an infinite time horizon (see Proposition \ref{momentbound}  for a detailed proof). This requires an additional assumption \ref{h1}. Moreover, the choice of the taming factors is crucial in the sense that the tamed coefficients should converge to the original coefficients with a desired rate.\\
{\bf Rate of convergence.} In the context of SDEs with superlinear coefficients, the strong convergence results of the tamed numerical schemes have been studied in depth in literature. One may refer to \cite{Kruse et al.}, \cite{hutzenthaler2012}, \cite{milsteinscheme}, \cite{eulerscheme}, \cite{SabanisAoAP}, \cite{WangGan} for the convergence resutls of tamed Euler and Milstein schemes in a finite time. In addition, Theorem 1 in \cite{SabanisYing} provides a strong convergence result in $\mathscr{L}^2$ of the tamed order 1.5 scheme. As mentioned in the introduction, while the aforementioned results focused on the convergence rates in finite time horizons, \cite{tula} considers a TULA algorithm which provides rate 1 in Wasserstein-2 distance and rate 1/2 in total variation. By extending the results in \cite{SabanisYing} and \cite{tula}, Theorem \ref{thmwasserstein} and Theorem \ref{thmtv} state that the convergence results of the HOLA algorithm \eqref{eqnscheme2} in Wasserstein-2 distance and in total variation can be improved to rate $1+\beta/2$  and rate 1 repectively. One notices that the assumptions \ref{h1} and \ref{h4} are the same as the assumptions in \cite{tula}, while the local H\"{o}lder condition \ref{h3} is the same as the assumption A-4 in \cite{SabanisYing}.\\
As for the SDEs with Lipschitz coefficients, \cite{dal16}, \cite{dal17user}, \cite{DM16} and \cite{DM17} provide convergence results in Wasserstein-2 distance and in total variation for the ULA algorithm. In addition, LMCO and LMCO' algorithms are considered in \cite{dal16} and \cite{dal17user} which make use of the Hessian of $U$, however, the rate of convergence is shown to have the same order as ULA in Wasserstein-2 distance. Under \ref{h4} - \ref{h7}, Theorem \ref{thmwassersteinlip} provides a convergence result for the scheme \eqref{eqnscheme2} in Wasserstein-2 distance, which is of order 1.5. It improves  existing results by imposing four times differentiability on the potential $U$ and an additional assumption \ref{h7}. \\ 
{\bf Non-asymptotic bounds and computational complexity.} The nonasymptotic bounds in total variation between the ULA algorithm and SDE \eqref{eqn:sde} are established in \cite{dal16}. Subsequently, improved results, including the Wasserstein-2 distance, are provided in \cite{dal17user}, \cite{DM16} and \cite{DM17} with better dependence on the dimension. Theorem \ref{thmwassersteinlip} in this article provides the non-asymptotic bound between the HOLA algorithm \eqref{eqnscheme2} and the target distribution $\pi$ in Wasserstein-2 distance  for the Lipschitz case. It shows that the dependence on dimension is $O(d^4 )$, and the number of iterations required to reach $\varepsilon$ percision level is given precisely by $n \geq \left((2\bar{C})^{\frac{1}{3}}/m\varepsilon^{\frac{2}{3}}\right)\log\left(4(|x-x^{\ast}|^2+d/m)/\varepsilon^2\right)$ with $\bar{C}=O(d^4)$. This implies that compared to results in \cite{dal17user} and \cite{DM16}, the HOLA algorithm \eqref{eqnscheme2} requires fewer steps to reach a suitably high precision level, i.e. for $\varepsilon < O(d^{-1})$. As for the computational complexity of the algorithm \eqref{eqnscheme2}, it shows in \cite{ag93} that the computational cost for the Hessian-vector product is not more expensive than evaluating the gradient. Moreover, although the computational cost for one iteration increases due to third derivatives of $U$, there are techniques which can be employed to reduce the computational cost dramatically, see \cite{guw2000}, \cite{guw08} and references therein.

\section{Local Lipschitz case}\label{locallipcase}
\subsection{Moment bounds}
It is a well-known result that by \ref{h1} and \ref{h3}, the SDE \eqref{eqn:sde} has a unique strong solution. One then needs to obtain moment bounds of the SDE \eqref{eqn:sde} and the numerical scheme \eqref{eqnscheme2} before considering the convergence results. 

By using Foster-Lyapunov conditions, one can obtain the exponential moment bounds for the solution of SDE \eqref{eqn:sde}. More concretely, the application of Theorem 1.1, 6.1 in \cite{RT96} and Theorem 2.2 in \cite{MT93} yields the following results.
\begin{prop} \label{sdemomentbound} Assume \ref{h1} and \ref{h3} are satisfied. For all $a>0$, there exists $b_a>0$, such that for all $x \in \mathbb{R}^d$,
\[
\mathscr{A}V_a(x) \leq -a V_a(x)+ab_a,
\]
and
\[
\sup_{t\geq 0} P_tV_a(x) \leq V_a(x)+b_a.
\]
Furthermore, there exist $ C_a >0$ and $\rho_a \in (0,1)$ such that for all $t>0$ and probability measures $\mu_0, \nu_0$ on $(\mathbb{R}^d, \mathcal{B}(\mathbb{R}^d))$ satisfying $\mu_0(V_a) +\nu_0(V_a)<+\infty$,
\[
\|\mu_0 P_t-\nu_0 P_t\|_{V_a} \leq C_a \rho_a^t \|\mu_0 - \nu_0\|_{V_a}, \quad \|\mu_0P_t-\pi\|_{V_a} \leq C_a \rho_a^t \mu_0(V_a).
\]
\end{prop}
\begin{proof}
One can refer to Proposition 1 in \cite{tula} for the detailed proof.
\end{proof}
The proposition below provides a uniform bound for exponential moments of the Markov chain $(\overline{X}_{k})_{k \geq 0}$.
\begin{prop}\label{momentbound} Assume \ref{h1} and \ref{h3} are satisfied. Then, there exist constants $b, c, M>0$, such that for all $x \in \mathbb{R}^d$ and $\gamma \in (0,1)$,
\[
R_{\gamma}V_c(x) \leq e^{-\frac{7}{3}c^2\gamma }V_c(x)+\gamma b \mathbf{1}_{\overline{\mathrm{B}}(0,M)}(x),
\]
and for all $n \in \mathbb{N}$
\[
R_{\gamma}^nV_c(x) \leq e^{-\frac{7}{3}c^2n\gamma }V_c(x)+\frac{3b}{7c^2}e^{\frac{7}{3}c^2\gamma }.
\]
Moreover, this guarantees that the Gaussian kernel $R_{\gamma}$ has a unique invariant measure $\pi_{\gamma}$ and $R_{\gamma}$ is geometrically ergodic w.r.t. $\pi_{\gamma}$.
\end{prop}
\begin{proof} We use the scheme \eqref{eqnscheme2} throughout the proof. First, one observes that by \ref{h1}, for $\gamma \in (0,1)$, the following holds
\begin{equation}\label{numericala1}
\liminf_{|x|\rightarrow +\infty}\frac{x}{|x|}\nabla U_{\gamma}(x) -\frac{\gamma}{2|x|}|\nabla U_{\gamma}(x)|^2 >0.
\end{equation}
Indeed, by \ref{h1}, there exist $M', \kappa >0$ such that for all $|x|\geq M'$, $x \in \mathbb{R}^d$, $x\nabla U(x) \geq \kappa |x||\nabla U(x)|$. Then, we have for all $|x|\geq M'$, $x \in \mathbb{R}^d$,
\begin{align*}
& \frac{x}{|x|}\nabla U_{\gamma}(x) -\frac{\gamma}{2|x|}|\nabla U_{\gamma}(x)|^2 \\
		& \geq \frac{1}{2|x|(1+\gamma^{3/2}|\nabla U(x)|^{3/2})^{2/3}}\left(2\kappa|x||\nabla U(x)| -\frac{\gamma|\nabla U(x)|^2}{(1+\gamma^{3/2}|\nabla U(x)|^{3/2})^{2/3}} \right)\\
		&  \geq \frac{|\nabla U(x)|}{2|x|(1+\gamma^{3/2}|\nabla U(x)|^{3/2})^{2/3}}\left(2\kappa|x| -\frac{\sqrt[3]{2}\gamma|\nabla U(x)|}{1+\gamma|\nabla U(x)|} \right)\\
		&  \geq \frac{|\nabla U(x)|}{2(1+\gamma^{3/2}|\nabla U(x)|^{3/2})^{2/3}}\left(2\kappa -\frac{\sqrt[3]{2} }{|x|}\right).
\end{align*}
Meanwhile, by \ref{h1}, there exist $M'', K>0$ such that for all $|x|\geq M''$, $x \in \mathbb{R}^d$, $|\nabla U|\geq K$. Note that $f(x) = x/(1+x^{3/2})^{2/3}$ is a non-decreasing function for all $x \geq 0$. Then, one obtains \eqref{numericala1}, since for all $x \in \mathbb{R}^d$, $|x| \geq \max(M', M'', \sqrt[3]{2}\kappa^{-1})$
\[
 \frac{x}{|x|}\nabla U_{\gamma}(x) -\frac{\gamma}{2|x|}|\nabla U_{\gamma}(x)|^2 \geq \frac{\kappa K}{2(1+\gamma^{3/2}K^{3/2})^{2/3}} >0.
\]
The function $f(x) = (1+|x|^2)^{1/2}$ is Lipschitz continuous with Lipschitz constant equal to 1. Let $ \overline{X}_0=x$, then for all $x \in \mathbb{R}^d$, applying log Sobolev inequality (see Proposition 5.5.1 in \cite{ledoux} and Appendix \ref{logsobproof} for a detailed proof) gives,
\begin{equation}\label{logsob}
R_{\gamma}V_a(x) =\mathbb{ E}_x(V_a(\overline{X}_1))\leq e^{\frac{7}{3}\gamma a^2}\exp\left\{a\mathbb{ E}((1+|\overline{X}_1|^2)^{1/2}|\overline{X}_0=x)\right\},
\end{equation}
which using Jensen's inequality yields
\begin{align} \label{expression1}
\begin{split}
&R_{\gamma}V_a(x) \\
&\leq e^{\frac{7}{3}\gamma a^2}  \exp\left\{a\left(1+\mathbb{ E}\left(\left.\left|\overline{X}_0+\mu_{\gamma}(\overline{X}_0)\gamma+\sigma_{\gamma}(\overline{X}_0)\sqrt{2\gamma}Z_1\right|^2\right|\overline{X}_0=x\right) \right)^{1/2}\right\}.
\end{split}
\end{align}
One calculates 
\begin{align}\label{sigma}
\begin{split}
&\mathbb{ E}\left[\left.\left|\sigma_{\gamma}(\overline{X}_0)\sqrt{2\gamma} Z_1\right|^2\right|\overline{X}_0=x\right]\\
&\qquad \leq 2\gamma \left|\sigma_{\gamma}(x)\right|^2\mathbb{ E}\left[\left| Z_1\right|^2\right] \\
									& \qquad \leq 2\gamma  d  +\frac{2\gamma^3}{3}\left|\nabla^2 U_{\gamma}(x)\right|^2 d+2\gamma^2\left|\nabla^2 U_{\gamma}(x)\right| d\\
									& \qquad\leq \frac{14}{3}d\gamma.
\end{split}
\end{align}
Then, by inserting \eqref{sigma} into \eqref{expression1}, one obtains
\begin{align}\label{prop1eqn1}
\begin{split}
R_{\gamma}V_a(x) &\leq e^{\frac{7}{3}\gamma a^2}  \exp\left\{a\left(1+A_{\gamma}(x) +\frac{14}{3}d\gamma \right)^{1/2}\right\},
\end{split}
\end{align}
where
\[
A_{\gamma}(x) =\left|x-\nabla U_{\gamma}(x)\gamma +\frac{\gamma^2}{2}\left(\left(\nabla^2 U\nabla U\right)_{\gamma}(x)-\vec{\Delta}(\nabla U)_{\gamma}(x)\right)\right|^2.
\]
Then, expanding the square yields
\begin{align*}
A_{\gamma}(x) 	& = |x|^2-2\gamma x\nabla U_{\gamma}(x)+\gamma^2\left|\nabla U_{\gamma}(x)\right|^2 -\gamma^2 x\vec{\Delta}(\nabla U)_{\gamma}(x)\\
					&\hspace{1em} +\frac{\gamma^4}{4}\left|\vec{\Delta}(\nabla U)_{\gamma}(x)\right|^2	+ \gamma^2x\left(\nabla^2 U\nabla U\right)_{\gamma}(x) \\
					&\hspace{1em} -\gamma^3 \nabla U_{\gamma}(x)\left(\nabla^2 U\nabla U\right)_{\gamma}(x) +\gamma^3 \nabla U_{\gamma}(x)\vec{\Delta}(\nabla U)_{\gamma}(x)\\
					&\hspace{1em} +\frac{\gamma^4}{4}\left|\left(\nabla^2 U\nabla U\right)_{\gamma}(x) \right|^2-\frac{\gamma^4}{2}\left(\nabla^2 U\nabla U\right)_{\gamma}(x)\vec{\Delta}(\nabla U)_{\gamma}(x).
\end{align*}
By \eqref{numericala1}, there exist $M_1, \kappa_1 >0$ such that for all $|x|\geq M_1$,
\[
x\nabla U_{\gamma}(x)  -\frac{\gamma}{2}|\nabla U_{\gamma}(x)|^2 >\kappa_1|x|.
\]
Thus, by using Remark \ref{tametermbd}, for all $|x|\geq  \max\{1, M_1\}$,
\begin{align*}
A_{\gamma}(x)+\frac{14}{3}d\gamma  & \leq |x|^2-2\gamma \kappa_1 |x|+\gamma^{3/2}+\frac{1}{4}\gamma^3\\
&\quad +3\gamma+2\gamma^{3/2}+\frac{1}{4}\gamma^2+\frac{1}{2}\gamma^{5/2}+\frac{14}{3}d\gamma \\
								& \leq |x|^2-2\gamma \kappa_1 |x|+\frac{35}{3}d\gamma.
\end{align*}
Denote by $M = \max\{1, M_1,\frac{35}{3}d(\kappa_1)^{-1}\}$, for all $x \in \mathbb{R}^d$, $|x|\geq M$,
\begin{align*}
A_{\gamma}(x)+\frac{14}{3}d\gamma   \leq |x|^2-\gamma\kappa_1 |x|.
\end{align*}
For $t \in [0,1]$, $(1-t)^{1/2} \leq 1-t/2$ and $g(x) = x/(1+x^2)^{1/2}$ is a non-decreasing function for all $x \geq 0$. Then, for all $x \in \mathbb{R}^d$, $|x|\geq M$  
\begin{align}\label{prop1eqn2}
\left(1+A_{\gamma}(x)+\frac{14}{3}d\gamma\right)^{1/2}  	& \leq \left(1+|x|^2\right)^{1/2}\left(1-\frac{7\gamma}{3}\frac{3\kappa_1 |x|}{7(1+|x|^2)} \right)^{1/2}  \nonumber \\
										&\leq  \left(1+|x|^2\right)^{1/2}-\frac{7\gamma}{3} \frac{3\kappa_1 M}{14(1+M^2)^{1/2}}.
\end{align}
By substituting \eqref{prop1eqn2} into \eqref{prop1eqn1} and completing the square, one obtains, for $|x|\geq M$,
\[
R_{\gamma}V_c(x)\leq e^{-\frac{7}{3}c^2\gamma}V_c(x),
\]
where 
\begin{equation}\label{constantc}
c = \frac{3\kappa_1 M}{28(1+M^2)^{1/2}}. 
\end{equation}
For the case $|x|\leq M$, by Remark \ref{tametermbd}, the following result can be obtained:
\[
A_{\gamma}(x) \leq |x|^2 + c_3\gamma (1+M)^{4\rho +2},
\]
where $c_3$ is a positive constant (that depends on $d$ and $L$). Then, by using $(1+s_1+s_2)^{1/2}\leq(1+s_1)^{1/2}+s_2/2$ for $s_1, s_2 \geq 0$,
\[
\left(1+A_{\gamma}(x)+\frac{14}{3}d\gamma\right)^{1/2} \leq (1+|x|^2)^{1/2}+\gamma\left(\frac{c_3}{2}(1+M)^{4\rho+2}+\frac{7d}{3}\right).
\]
Thus,
\[
R_{\gamma}V_c(x) \leq e^{\theta\gamma }V_c(x),
\]
where $\theta = \frac{7}{3}c^2+c(\frac{c_3}{2}(1+M)^{4\rho +2}+\frac{7}{3}d)$. Moreover, for $|x| \leq M$, 
\[
R_{\gamma}V_c(x) -e^{-\frac{7}{3}c^2\gamma }V_c(x) \leq e^{\theta \gamma}(1-e^{-\gamma\left(\frac{7}{3}c^2+\theta\right)})V_c(x) \leq \gamma e^{\theta \gamma}\left(\frac{7}{3}c^2+\theta\right)V_c(x).
\]
Denote by $b=e^{(\theta \gamma+c\sqrt{1+M^2})}\left(\frac{7}{3}c^2+\theta\right)$, one obtains
\[
R_{\gamma}V_c(x) \leq e^{-\frac{7}{3}c^2\gamma }V_c(x)+\gamma b \mathbf{1}_{\overline{\mathrm{B}}(0,M)}(x).
\]
Then by induction, for all $n \in \mathbb{N}$ and $x \in \mathbb{R}$
\begin{align*}
R_{\gamma}^nV_c(x) &\leq e^{-\frac{7}{3}c^2n\gamma }V_c(x)+\frac{1-e^{-\frac{7}{3}c^2n\gamma }}{1-e^{-\frac{7}{3}c^2\gamma }}\gamma b \\
								&\leq e^{-\frac{7}{3}c^2n\gamma }V_c(x)+\frac{3b}{7dc^2}e^{\frac{7}{3}c^2\gamma },
\end{align*}
the last inequality holds since $1-e^{-\frac{7}{3}c^2\gamma }= \int_0^{\gamma}\frac{7}{3}c^2e^{-\frac{7}{3}c^2s}\,ds\geq \frac{7}{3}c^2\gamma e^{-\frac{7}{3}c^2\gamma}$. Finally, since any compact set on $\mathbb{R}^d$ is accessible and small for $R_{\gamma}$, then by section 3.1 in \cite{RT96} and Theorem 15.0.1 in \cite{MTmarkovchain}, for all $\gamma \in (0,1)$, $R_{\gamma}$ has a unique invariant measure $\pi_{\gamma}$ and it is geometrically ergodic w.r.t. $\pi_{\gamma}$.
\end{proof}
The results in Proposition \ref{sdemomentbound} and \ref{momentbound} provide exponential moment bounds for the solution of SDE \eqref{eqn:sde} and the scheme \eqref{eqnscheme2}, which enable us to consider the total variation and Wasserstein distance between the target distribution $\pi$ and the $n$-th iteration of the MCMC algorithm.

\subsection{Proof of Theorem \ref{thmwasserstein}}
In order to obtain the convergence rate in Wasserstein distance, the assumption \ref{h4} is needed, which assumes the convexity of $U$. We consider the linear interpolation of the scheme \eqref{eqnscheme} given by
\begin{align}\label{eqnintp}
\begin{split}
&\bar{x}_{t} = \bar{x}_0-\int_0^{t} \nabla\tilde{ U}_{\gamma}(s,\bar{x}_{\floor{s/\gamma}\gamma}) \, ds+ \sqrt{2}w_t,
\end{split}
\end{align}
for all $t \geq 0$, where
\[
\nabla\tilde{ U}_{\gamma}(s,\bar{x}_{\floor{s/\gamma}\gamma}) =  \nabla U_{\gamma}(\bar{x}_{\floor{s/\gamma}\gamma})+\nabla U_{1,\gamma}(s, \bar{x}_{\floor{s/\gamma}\gamma}) +\nabla U_{2,\gamma}(s,\bar{x}_{\floor{s/\gamma}\gamma}),
\]
with
\[
\nabla U_{1,\gamma}(s,\bar{x}_{\floor{s/\gamma}\gamma}) =- \int_{\floor{s/\gamma}\gamma}^s\left( (\nabla^2 U\nabla U)_{\gamma}(\bar{x}_{\floor{s/\gamma}\gamma}) -\vec{\Delta}(\nabla U)_{\gamma}(\bar{x}_{\floor{s/\gamma}\gamma})\right)\, dr, 
\]
\[
\nabla U_{2,\gamma}(s,\bar{x}_{\floor{s/\gamma}\gamma})=\sqrt{2} \int_{\floor{s/\gamma}\gamma}^s \nabla^2 U_{\gamma}(\bar{x}_{\floor{s/\gamma}\gamma})\, dw_r.
\]
Note that the linear interpolation \eqref{eqnintp} and the scheme \eqref{eqnscheme} coincide at grid points, i.e. for any $n \in \mathbb{N}$, $X_n = \bar{x}_{n \gamma}$. Let $(\mathscr{F}_{t})_{t \geq 0}$ be a filtration associated with $(w_t)_{t \geq 0}$. For any $n \in \mathbb{N}$, denote by $\mathbb{E}^{\mathscr{F}_{n\gamma}}[\cdot]$ the expectation conditional on $\mathscr{F}_{n\gamma}$.
\begin{lemma}\label{rate1}
Assume \ref{h1} and \ref{h3} are satisfied. Then, there exists a constant $C>0$ such that for all $p>0$, $\gamma \in (0,1)$, $n \in \mathbb{N}$, and $t \in [n\gamma, (n+1)\gamma)$,
\[
\mathbb{E}^{\mathscr{F}_{n\gamma}}\left[|\nabla U_{1,\gamma}(t,\bar{x}_{n\gamma})|^p\right] \leq C\gamma^{p}V_c(\bar{x}_{n\gamma}), 
\]
\[
\mathbb{E}^{\mathscr{F}_{n\gamma}}\left[|\nabla U_{2,\gamma}(t,\bar{x}_{n\gamma})|^p\right] \leq C\gamma^{\frac{p}{2}}V_c(\bar{x}_{n\gamma}).
\]
\end{lemma}
\begin{proof}
Consider a polynomial function $f(|x|) \in C_{poly}(\mathbb{R}_+, \mathbb{R}_+)$, then there exists a constant $C>0$ such that for all $x \in \mathbb{R}^d$, $f(|x|) \leq CV_c(x)$. For $p > 1$, by applying H\"{o}lder's inequality and Remark \ref{poly}, one obtains
\begin{align*}
&\mathbb{E}^{\mathscr{F}_{n\gamma}}\left[|\nabla U_{1,\gamma}(t,\bar{x}_{n\gamma})|^p\right]\\
		 &= \mathbb{E}^{\mathscr{F}_{n\gamma}}\left[\left|- \int_{n\gamma}^t\left( (\nabla^2 U\nabla U)_{\gamma}(\bar{x}_{n\gamma}) -\vec{\Delta}(\nabla U)_{\gamma}(\bar{x}_{n\gamma})\right)\, dr\right|^p\right]\\
		 & \leq C\gamma^{p-1} \int_{n\gamma}^t \mathbb{E}^{\mathscr{F}_{n\gamma}}\left[\left| (\nabla^2 U\nabla U)_{\gamma}(\bar{x}_{n\gamma})\right|^p\right]\, dr\\
		 &\hspace{1em}+ C\gamma^{p-1} \int_{n\gamma}^t \mathbb{E}^{\mathscr{F}_{n\gamma}}\left[\left|\vec{\Delta}(\nabla U)_{\gamma}(\bar{x}_{n\gamma})\right|^p\right]\, dr\\
		 & \leq C\gamma^p V(\bar{x}_{n\gamma}),
\end{align*}
The second inequality can be proved using similar arguments. For the case $0<p\leq 1$, Jensen's inequality is used to obtain the desired result.
\end{proof}
\begin{lemma}\label{rate2}
Assume \ref{h1} and \ref{h3} are satisfied. Then, there exists a constant $C>0$ such that for all $p>0$, $\gamma \in (0,1)$, $n \in \mathbb{N}$, and $t \in [n\gamma, (n+1)\gamma)$,
\[
\mathbb{E}^{\mathscr{F}_{n\gamma}}\left[|\bar{x}_t -\bar{x}_{n\gamma} |^p\right] \leq C\gamma^{\frac{p}{2}}V_c(\bar{x}_{n\gamma}),
\]
\[
\mathbb{E}^{\mathscr{F}_{n\gamma}}\left[|x_t -x_{n\gamma} |^p\right] \leq C\gamma^{\frac{p}{2}}V_c(x_{n\gamma}).
\]
\end{lemma}
\begin{proof}
For $p > 1$, by using H\"{o}lder's inequality, Remark \ref{poly} and Lemma \ref{rate1}, we have
\begin{align*}
&\mathbb{E}^{\mathscr{F}_{n\gamma}}\left[|\bar{x}_t -\bar{x}_{n\gamma} |^p\right] \\
		& = \mathbb{E}^{\mathscr{F}_{n\gamma}}\left[\left|-\int_{n\gamma}^{t} \nabla\tilde{ U}_{\gamma}(s,\bar{x}_{n\gamma}) \, ds+ \sqrt{2}\int_{n\gamma}^{t}\,dw_s \right|^p\right]\\
		& \leq C \gamma^{p-1} \int_{n\gamma}^{t} \mathbb{E}^{\mathscr{F}_{n\gamma}}\left[\left| \nabla U_{\gamma}(\bar{x}_{n\gamma}) + \nabla U_{1, \gamma}(s,\bar{x}_{n\gamma})+ \nabla U_{2, \gamma}(s,\bar{x}_{n\gamma})\right|^p\right]\,ds+ C\gamma^{\frac{p}{2}}\\
		& \leq C\gamma^{\frac{p}{2}}V_c(\bar{x}_{n\gamma}).
\end{align*}
For the case $0<p \leq 1$, one can use Jensen's inequality to obtain
\begin{align*}
\mathbb{E}^{\mathscr{F}_{n\gamma}}\left[|\bar{x}_t -\bar{x}_{n\gamma} |^p\right] 
		& \leq \left(\mathbb{E}^{\mathscr{F}_{n\gamma}}\left|\int_{n\gamma}^{t} \nabla\tilde{ U}_{\gamma}(s,\bar{x}_{n\gamma})  \, ds +\sqrt{2}\int_{n\gamma}^{t}\,dw_s\right| \right)^p\\
		& \leq C\gamma^{\frac{p}{2}}V_c(\bar{x}_{n\gamma}),
\end{align*}
Similarly, for $p>1$, by using H\"{o}lder's inequality, one obtains
\begin{align*}
\mathbb{E}^{\mathscr{F}_{n\gamma}}\left[|x_t -x_{n\gamma} |^p\right] 
		& = \mathbb{E}^{\mathscr{F}_{n\gamma}}\left[\left|-\int_{n\gamma}^{t} \nabla U(x_s) \, ds+ \sqrt{2}\int_{n\gamma}^{t}\,dw_s \right|^p\right]\\
		& \leq C \gamma^{p-1} \int_{n\gamma}^{t} \mathbb{E}^{\mathscr{F}_{n\gamma}}\left(1+|x_s|^{p(\rho+\beta)}\right)\,ds+ C\gamma^{\frac{p}{2}}\\
		& \leq C\gamma^{\frac{p}{2}}V_c(x_{n\gamma}),
\end{align*}
where the last inequality holds due to Proposition \ref{sdemomentbound}. The case $p \in (0,1]$ follows from the application of Jensen's inequality.
\end{proof}
\begin{lemma}\label{rate4terms}
Assume \ref{h1} and \ref{h3} are satisfied. Then, there exists a constant $C>0$ such that for all $\gamma \in (0,1)$, $n \in \mathbb{N}$, and $t \in [n\gamma, (n+1)\gamma)$,
\begin{align*}
\mathbb{E}^{\mathscr{F}_{n\gamma}}\left[|\nabla U(\bar{x}_t)- \nabla U(\bar{x}_{n\gamma})- \nabla U_{1,\gamma}(t,\bar{x}_{n\gamma})-\nabla U_{2,\gamma}(t,\bar{x}_{n\gamma})|^2\right] \leq C\gamma^2V_c(\bar{x}_{n\gamma}).
\end{align*}
\end{lemma}
\begin{proof}
For any $t \in [n\gamma, (n+1)\gamma)$, applying It\^o's formula to $\nabla U(\bar{x}_t)$ gives, almost surely
\begin{align*}
&\nabla U(\bar{x}_t)- \nabla U(\bar{x}_{n\gamma}) \\
									&= -\int_{n\gamma}^t \left(\nabla^2U(\bar{x}_r)\nabla\tilde{ U}_{\gamma}(r,\bar{x}_{n\gamma})-\vec{\Delta}(\nabla U)(\bar{x}_r)\right)\,dr +\sqrt{2}\int_{n\gamma}^t\nabla^2 U(\bar{x}_r)\,dw_r\\
									&=- \int_{n\gamma}^t\left(\nabla^2 U(\bar{x}_r)-\nabla^2 U(\bar{x}_{n\gamma})\right)\nabla U_{\gamma}(\bar{x}_{n\gamma})\,dr-\int_{n\gamma}^t \nabla^2U(\bar{x}_{n\gamma})\nabla U_{\gamma}(\bar{x}_{n\gamma})\,dr\\
									& \hspace{1em} -\int_{n\gamma}^t\nabla^2U(\bar{x}_r)(\nabla U_{1,\gamma}(r,\bar{x}_{n\gamma})+\nabla U_{2,\gamma}(r,\bar{x}_{n\gamma}))\,dr\\
									& \hspace{1em} +\sqrt{2}\int_{n\gamma}^t\left(\nabla^2 U(\bar{x}_r)-\nabla^2 U(\bar{x}_{n\gamma})\right)\,dw_r+\sqrt{2}\int_{n\gamma}^t\nabla^2 U(\bar{x}_{n\gamma})\,dw_r\\
									& \hspace{1em} +\int_{n\gamma}^t\left(\vec{\Delta}(\nabla U)(\bar{x}_r)-\vec{\Delta}(\nabla U)(\bar{x}_{n\gamma})\right)\,dr +\int_{n\gamma}^t\vec{\Delta}(\nabla U)(\bar{x}_{n\gamma})\,dr.
\end{align*}
By substracting $\nabla U_{1,\gamma}(t,\bar{x}_{n\gamma})$, $\nabla U_{2,\gamma}(t,\bar{x}_{n\gamma})$, squaring both sides and taking conditional expectation yields, 
\begin{equation}\label{rateeqn3}
 \mathbb{E}^{\mathscr{F}_{n\gamma}}\left[ \left| \nabla U(\bar{x}_t)- \nabla U(\bar{x}_{n\gamma})- \nabla U_{1,\gamma}(t,\bar{x}_{n\gamma})-\nabla U_{2,\gamma}(t,\bar{x}_{n\gamma})\right|^2\right] \leq C\sum_{i=1}^5 G_i(t).
\end{equation}
where
\begin{align*}
G_1(t)	& =  \mathbb{E}^{\mathscr{F}_{n\gamma}}\left[\left|-\int_{n\gamma}^t\left(\nabla^2 U(\bar{x}_r)-\nabla^2 U(\bar{x}_{n\gamma})\right)\nabla U_{\gamma}(\bar{x}_{n\gamma})\,dr\right|^2\right] ,\\
G_2(t)		&= \mathbb{E}^{\mathscr{F}_{n\gamma}}\left[\left|-\int_{n\gamma}^t\nabla^2U(\bar{x}_r)\left(\nabla U_{1,\gamma}(r,\bar{x}_{n\gamma})+\nabla U_{2,\gamma}(r,\bar{x}_{n\gamma})\right)\,dr\right|^2\right],\\
G_3(t)	&= \mathbb{E}^{\mathscr{F}_{n\gamma}}\left[\left|\sqrt{2}\int_{n\gamma}^t\left(\nabla^2 U(\bar{x}_r)-\nabla^2 U(\bar{x}_{n\gamma})\right)\,dw_r\right|^2\right],\\
G_4(t)	&= \mathbb{E}^{\mathscr{F}_{n\gamma}}\left[\left|\int_{n\gamma}^t\left(\vec{\Delta}(\nabla U)(\bar{x}_r)-\vec{\Delta}(\nabla U)(\bar{x}_{n\gamma})\right)\,dr\right|^2\right],\\
G_5(t) 	& = \mathbb{E}^{\mathscr{F}_{n\gamma}}\left[\left(|\nabla^2U(\bar{x}_{n\gamma})||\nabla U(\bar{x}_{n\gamma})|^2\gamma^2+|\bar{x}_{n\gamma}||\nabla^2U(\bar{x}_{n\gamma})|^2|\nabla U(\bar{x}_{n\gamma})|^2\gamma^2\right.\right.\\
		& \hspace{1em} \left.\left.+\gamma^{3/2}|\bar{x}_{n\gamma}||\vec{\Delta}(\nabla U)(\bar{x}_{n\gamma})|^2+\sqrt{2}\gamma|\nabla^2 U(\bar{x}_{n\gamma})|^2|w_t-w_{n\gamma}|\right)^2\right].
\end{align*} 
By using Cauchy-Schwarz inequality, Proposition \ref{momentbound}, Remark \ref{poly} and Lemma \ref{rate2}, one obtains
\begin{align}\label{rateg1} 
\begin{split}
G_1(t)		& \leq \gamma\int_{n\gamma}^t\mathbb{E}^{\mathscr{F}_{n\gamma}}\left[|(\nabla^2 U(\bar{x}_r)-\nabla^2 U(\bar{x}_{n\gamma}))\nabla U_{\gamma}(\bar{x}_{n\gamma})|^2\right]\,dr\\
	& \leq  C\gamma\int_{n\gamma}^t\mathbb{E}^{\mathscr{F}_{n\gamma}}\left[(1+|\bar{x}_r|+|\bar{x}_{n\gamma}|)^{4\rho-4+4\beta}|\bar{x}_r-\bar{x}_{n\gamma}|^2\right]\,dr\\
	& \leq  C\gamma\int_{n\gamma}^t\sqrt{\mathbb{E}^{\mathscr{F}_{n\gamma}}\left[V_c(\bar{x}_r)+V_c(\bar{x}_{n\gamma})\right]\mathbb{E}^{\mathscr{F}_{n\gamma}}\left[|\bar{x}_r-\bar{x}_{n\gamma}|^4\right]}\,dr\\
	& \leq C\gamma^3 V_c(\bar{x}_{n\gamma}).
\end{split}
\end{align}
Similarly, applying Cauchy-Schwarz inequality, Proposition \ref{momentbound} and Remark \ref{poly} yield
\begin{align}\label{rateg2}
\begin{split}
G_2(t)		& \leq \gamma\int_{n\gamma}^t\mathbb{E}^{\mathscr{F}_{n\gamma}}\left[|\nabla^2U(\bar{x}_r)(\nabla U_{1,\gamma}(r,\bar{x}_{n\gamma})+\nabla U_{2,\gamma}(r,\bar{x}_{n\gamma}))|^2\right]\,dr\\
	& \leq  C\gamma\int_{n\gamma}^t\mathbb{E}^{\mathscr{F}_{n\gamma}}\left[(1+|\bar{x}_r|)^{2\rho-2+2\beta}|\nabla U_{1,\gamma}(r,\bar{x}_{n\gamma})+\nabla U_{2,\gamma}(r,\bar{x}_{n\gamma})|^2\right]\,dr\\
	& \leq  C\gamma\int_{n\gamma}^t\sqrt{\mathbb{E}^{\mathscr{F}_{n\gamma}}\left[V_c(\bar{x}_r)\right]\mathbb{E}^{\mathscr{F}_{n\gamma}}\left[|\nabla U_{1,\gamma}(r,\bar{x}_{n\gamma})|^4+|\nabla U_{2,\gamma}(r,\bar{x}_{n\gamma})|^4\right]}\,dr\\
	& \leq C\gamma^3 V_c(\bar{x}_{n\gamma}),
\end{split}
\end{align}
where the last inequality is obtained by applying Lemma \ref{rate1}. Moreover, one obtains 
\begin{align*}
\begin{split}
G_3(t)	& \leq C\int_{n\gamma}^t\mathbb{E}^{\mathscr{F}_{n\gamma}}\left[|\nabla^2 U(\bar{x}_r)-\nabla^2 U(\bar{x}_{n\gamma})|^2\right]\,dr\\
	& \leq  C\int_{n\gamma}^t\mathbb{E}^{\mathscr{F}_{n\gamma}}\left[(1+|\bar{x}_r|+|\bar{x}_{n\gamma}|)^{2\rho-4+2\beta}|\bar{x}_r-\bar{x}_{n\gamma}|^2\right]\,dr\\
	& \leq  C\int_{n\gamma}^t\sqrt{\mathbb{E}^{\mathscr{F}_{n\gamma}}\left[V_c(\bar{x}_r)+V_c(\bar{x}_{n\gamma})\right]\mathbb{E}^{\mathscr{F}_{n\gamma}}\left[|\bar{x}_r-\bar{x}_{n\gamma}|^4\right]}\,dr\\
	& \leq C\gamma^2 V_c(\bar{x}_{n\gamma}).
\end{split}
\end{align*}
Furthermore, using Cauchy-Schwarz inequality, Proposition \ref{momentbound}, Lemma \ref{rate2} and \ref{h3} yield
\begin{align}\label{rateg4}
\begin{split}
G_4(t)	& \leq \gamma\int_{n\gamma}^t\mathbb{E}^{\mathscr{F}_{n\gamma}}\left[\left|\vec{\Delta}(\nabla U)(\bar{x}_r)-\vec{\Delta}(\nabla U)(\bar{x}_{n\gamma})\right|^2\right]\,dr\\
	& \leq  C\gamma\int_{n\gamma}^t\mathbb{E}^{\mathscr{F}_{n\gamma}}\left[(1+|\bar{x}_r|+|\bar{x}_{n\gamma}|)^{2\rho-4}|\bar{x}_r-\bar{x}_{n\gamma}|^{2\beta}\right]\,dr\\
	& \leq  C\gamma\int_{n\gamma}^t\sqrt{\mathbb{E}^{\mathscr{F}_{n\gamma}}\left[V_c(\bar{x}_r)+V_c(\bar{x}_{n\gamma})\right]\mathbb{E}^{\mathscr{F}_{n\gamma}}\left[|\bar{x}_r-\bar{x}_{n\gamma}|^{4\beta}\right]}\,dr\\
	& \leq C\gamma^{2+\beta} V_c(\bar{x}_{n\gamma}).
\end{split}
\end{align}
The estimate of $G_5(t)$ can be obtained by straightforwad calculations, which implies $G_5(t) \leq C\gamma^3 V_c(\bar{x}_{n\gamma})$. Therefore, 
\[
 \mathbb{E}^{\mathscr{F}_{n\gamma}}\left[ |\nabla U(\bar{x}_t)- \nabla U(\bar{x}_{n\gamma})- \nabla U_{1,\gamma}(t,\bar{x}_{n\gamma})-\nabla U_{2,\gamma}(t,\bar{x}_{n\gamma})|^2\right] \leq C\gamma^2 V_c(\bar{x}_{n\gamma}).
\]
\end{proof}
For any $x, \bar{x} \in \mathbb{R}^d$, denote by $M(x, \bar{x})$ a matrix whose $(i,j)$-th entry is $\sum_{k=1}^d \frac{\partial^3 U(\bar{x})}{\partial x^{(i)} \partial x^{(j)}\partial x^{(k)}}(x^{(k)}-\bar{x}^{(k)})$. One then obtains the following results.
\begin{lemma}\label{mvt} Assume \ref{h3} holds. Then, there exists a constant \(C>0\) such that for any \(x, \bar{x} \in \mathbb{R}^d\), and \(i = 1, \dots,d\), 
\[
\left|\nabla^2 U(x)-\nabla^2 U(\bar{x}) - M(x,\bar{x}) \right|\leq  \sqrt{d}L(1+|x|+|\bar{x}|)^{\rho-2 }|x-\bar{x}|^{1+\beta}.
\] 
\end{lemma}
\begin{proof} 
Denote by $g(t) = \nabla^2 U(tx+(1-t)\bar{x})$, for any $x, \bar{x} \in \mathbb{R}^d$ and $t \in [0,1]$. One notes that for any $i,j = 1, \dots,d$,
\begin{align*}
& \nabla^2 U^{(i,j)}(x) -  \nabla^2 U^{(i,j)}(\bar{x}) -  M^{(i,j)}(x,\bar{x})\\
&= \int_0^1 \nabla ( \nabla^2 U^{(i,j)})(tx+(1-t)\bar{x})(x- \bar{x})\, dt - \nabla ( \nabla^2 U^{(i,j)})( \bar{x})(x- \bar{x})
\end{align*}
One obtains that by Cauchy-Schwarz inequality and \ref{h3}
\begin{align*}
&\left|\nabla^2 U(x)-\nabla^2 U(\bar{x}) - M(x,\bar{x}) \right| \\
&\leq \int_0^1\left|\sum_{k=1}^d \left(\nabla^2 (\nabla U)^{(k)}(tx+(1-t)\bar{x})-\nabla^2 (\nabla U)^{(k)}(\bar{x})\right)(x^{(k)}-\bar{x}^{(k)})\right|\,dt\\
& \leq  \int_0^1\left|\left(\sum_{k=1}^d \left|\nabla^2 (\nabla U)^{(k)}(tx+(1-t)\bar{x})-\nabla^2 (\nabla U)^{(k)}(\bar{x})\right|^2\right)^{1/2}\right|\,dt|x-\bar{x}|\\
&\leq \sqrt{d}L(1+|x|+|\bar{x}|)^{\rho -2}|x-\bar{x}|^{1+\beta}.
\end{align*} 
\end{proof}
\begin{lemma}\label{rateM}
Assume \ref{h1} and \ref{h3} are satisfied. Then, there exists a constant $C>0$ such that for all $\gamma \in (0,1)$, $n \in \mathbb{N}$, and $t \in [n\gamma, (n+1)\gamma)$,
\[
\mathbb{E}^{\mathscr{F}_{n\gamma}}\left[\left|\int_{n\gamma}^{t}M(\bar{x}_r,\bar{x}_{n\gamma})\,dw_r\right|^2\right] \leq C\gamma^2V_c(\bar{x}_{n\gamma}).
\] 
\end{lemma}
\begin{proof} By using conditional It\^o's isometry and Lemma \ref{rate2}, one obtains,
\begin{align*} 
&\mathbb{E}^{\mathscr{F}_{n\gamma}}\left[\left|\int_{n\gamma}^{t}M(\bar{x}_r,\bar{x}_{n\gamma})\,dw_r\right|^2\right] \\
 &\leq C\mathbb{E}^{\mathscr{F}_{n\gamma}}\left[\int_{n\gamma}^{t}\left|M(\bar{x}_r,\bar{x}_{n\gamma})\right|^2\,dr\right] \\
& =C\mathbb{E}^{\mathscr{F}_{n\gamma}}\left[\int_{n\gamma}^{t}\left( \sum_{i,j=1}^d \left|\sum_{k=1}^d \frac{\partial^3 U(\bar{x}_{n\gamma})}{\partial x^{(i)} \partial x^{(j)}\partial x^{(k)}}(\bar{x}_r^{(k)}-\bar{x}_{n\gamma}^{(k)})\right|^2\right) \,dr\right] \\
& \leq C\int_{n\gamma}^t \mathbb{E}^{\mathscr{F}_{n\gamma}}\left[ (1+|\bar{x}_{n\gamma}|)^{2(\rho-2+\beta)}|\bar{x}_r-\bar{x}_{n\gamma}|^2\right] \,dr\\
& \leq C\gamma^2 V_c(\bar{x}_{n\gamma}).
\end{align*}
\end{proof}
\begin{lemma}\label{problematicterm}
Assume \ref{h1} and \ref{h3} are satisfied. Then, there exists a constant $C>0$ such that for all $\gamma \in (0,1)$, $n \in \mathbb{N}$, and $t \in [n\gamma, (n+1)\gamma)$,
\begin{align*}
&\mathbb{E}^{\mathscr{F}_{n\gamma}}\left[\int_{n\gamma}^t(\nabla U(x_r)-\nabla U(\bar{x}_r))\,dr\int_{n\gamma}^tM(\bar{x}_r,\bar{x}_{n\gamma}) \,dw_r\right]\\
&\leq C\gamma^3(V_c(\bar{x}_{n\gamma})+V_c(x_{n\gamma})).
\end{align*}
\end{lemma}
\begin{proof} 
For any $t \in [n\gamma, (n+1)\gamma)$, one observes that
\begin{align}\label{problematicsp}
\begin{split}
&\mathbb{E}^{\mathscr{F}_{n\gamma}}\left[\int_{n\gamma}^t(\nabla U(x_r)-\nabla U(\bar{x}_r))\,dr\int_{n\gamma}^tM(\bar{x}_r,\bar{x}_{n\gamma}) \,dw_r\right] \\
& = \mathbb{E}^{\mathscr{F}_{n\gamma}}\left[\int_{n\gamma}^t\left\{\nabla U(x_r)-\nabla U(x_{n\gamma})-(\nabla U(\bar{x}_r)-\nabla U(\bar{x}_{n\gamma})) \right. \right.\\
&\quad\left.\left. - \sqrt{2}\int_{n\gamma}^r\nabla^2 U(x_{n\gamma}) \,dw_s  + \sqrt{2}\int_{n\gamma}^r\nabla^2 U(\bar{x}_{n\gamma}) \,dw_s\right\}\,dr \int_{n\gamma}^tM(\bar{x}_r, \bar{x}_{n\gamma})\,dw_r\right]\\
&\hspace{1em} +  \sqrt{2}\mathbb{E}^{\mathscr{F}_{n\gamma}}\left[\int_{n\gamma}^t\int_{n\gamma}^r(\nabla^2 U(x_{n\gamma})- \nabla^2U(\bar{x}_{n\gamma})) \,dw_s\,dr\int_{n\gamma}^tM(\bar{x}_r, \bar{x}_{n\gamma})\,dw_r\right].
\end{split}
\end{align}
The second term in \eqref{problematicsp} can be rewritten as
\begin{align*}
& \sqrt{2} \mathbb{E}^{\mathscr{F}_{n\gamma}}\left[\int_{n\gamma}^t\int_{n\gamma}^r(\nabla^2 U(x_{n\gamma})- \nabla^2U(\bar{x}_{n\gamma})) \,dw_s\,dr\int_{n\gamma}^tM(\bar{x}_r, \bar{x}_{n\gamma})\,dw_r\right]\\
& = \sqrt{2} \mathbb{E}^{\mathscr{F}_{n\gamma}}\left[\sum_{i=1}^d \int_{n\gamma}^t\sum_{l=1}^d\int_{n\gamma}^r(\nabla^2 U^{(i,l)}(x_{n\gamma})- \nabla^2U^{(i,l)}(\bar{x}_{n\gamma})) \,dw^{(l)}_s\,dr \right.\\
& \left.\times \sum_{j=1}^d\int_{n\gamma}^t\sum_{k=1}^d \frac{\partial^3 U(\bar{x}_{n\gamma})}{\partial x^{(i)} \partial x^{(j)}\partial x^{(k)}}\left(-\int_{n\gamma}^r\nabla\tilde{U}^{(k)}_{\gamma}(s,\bar{x}_{n\gamma})\,ds +\sqrt{2}\int_{n\gamma}^r \,dw^{(k)}_s\right)\,dw^{(j)}_r\right]\\
&\leq C \gamma^3 (V_c(x_{n\gamma})+V_c(\bar{x}_{n\gamma})).
\end{align*}
where the last inequality holds due to Cauchy-Schwarz inequality, Lemma \ref{rate1}, Proposition \ref{sdemomentbound}, \ref{momentbound} and the fact that for any $i, l, j, k = 1, \dots, d$
\begin{align*}
 &\mathbb{E}^{\mathscr{F}_{n\gamma}}\left[ \int_{n\gamma}^t\int_{n\gamma}^r(\nabla^2 U^{(i,l)}(x_{n\gamma})- \nabla^2U^{(i,l)}(\bar{x}_{n\gamma})) \,dw^{(l)}_s\,dr \right.\\
 &\hspace{5em}\times \left. \int_{n\gamma}^t \frac{\partial^3 U(\bar{x}_{n\gamma})}{\partial x^{(i)} \partial x^{(j)}\partial x^{(k)}} \int_{n\gamma}^r\sqrt{2} \,dw^{(k)}_s\,dw^{(j)}_r\right]=0.
\end{align*}
Then, to estimate the first term of \eqref{problematicsp}, one applies It\^o's formula to $\nabla U(x_r)$ and $\nabla U(\bar{x}_r)$ to obtain, almost surely
\begin{align}\label{problematicito}
\begin{split}
& \nabla U(x_r)-\nabla U(x_{n\gamma})-(\nabla U(\bar{x}_r)-\nabla U(\bar{x}_{n\gamma}))\\
&\quad- \sqrt{2}\int_{n\gamma}^r\nabla^2 U(x_{n\gamma}) \,dw_s+ \sqrt{2}\int_{n\gamma}^r\nabla^2 U(\bar{x}_{n\gamma}) \,dw_s \\
&=-\int_{n\gamma}^r \left(\nabla^2U(x_s) \nabla U(x_s) - \vec{\Delta}(\nabla U)(x_s)\right)\,ds\\
&\quad+\sqrt{2}\int_{n\gamma}^r (\nabla^2 U(x_s)-\nabla^2 U(x_{n\gamma}))\,dw_s\\
& \quad +\int_{n\gamma}^r \left( \nabla^2 U(\bar{x}_s) \nabla \tilde{U}_{\gamma}(s, \bar{x}_{n\gamma}) - \vec{\Delta}(\nabla U)(\bar{x}_s)\right)\,ds\\
&\quad -\sqrt{2}\int_{n\gamma}^r (\nabla^2 U(\bar{x}_s)-\nabla^2 U(\bar{x}_{n\gamma}))\,dw_s.
\end{split}
\end{align}
By using Cauchy-Schwarz inequality and Lemma \ref{rateM}, equation \eqref{problematicsp} yields
\begin{align*}
& \mathbb{E}^{\mathscr{F}_{n\gamma}}\left[\int_{n\gamma}^t(\nabla U(x_r)-\nabla U(\bar{x}_r))\,dr\int_{n\gamma}^tM(\bar{x}_r,\bar{x}_{n\gamma})\,dw_r\right] \\
& \leq \sqrt{C\gamma^2 V_c(\bar{x}_{n\gamma})} \left(\mathbb{E}^{\mathscr{F}_{n\gamma}}\left[\gamma \int_{n\gamma}^t|\nabla U(x_r)-\nabla U(x_{n\gamma})-(\nabla U(\bar{x}_r)-\nabla U(\bar{x}_{n\gamma}))\right.\right.\\
&\quad \left.\left.- \sqrt{2}\int_{n\gamma}^r\nabla^2 U(x_{n\gamma}) \,dw_s+ \sqrt{2}\int_{n\gamma}^r\nabla^2 U(\bar{x}_{n\gamma})\,dw_s|^2\,dr\right]\right)^{1/2}\\
& \quad + C\gamma^3(V_c(\bar{x}_{n\gamma})+V_c(x_{n\gamma})).
\end{align*}
Then, by taking into consideration \eqref{problematicito}, and by applying Proposition \ref{sdemomentbound} and \ref{momentbound}, one obtains
\begin{align*}
& \mathbb{E}^{\mathscr{F}_{n\gamma}}\left[\int_{n\gamma}^t(\nabla U(x_r)-\nabla U(\bar{x}_r))\,dr\int_{n\gamma}^tM(\bar{x}_r,\bar{x}_{n\gamma})\,dw_r\right] \\
& \leq \sqrt{C\gamma^2 V_c(\bar{x}_{n\gamma})} \\
&\quad \times\left(\mathbb{E}^{\mathscr{F}_{n\gamma}}\left[\gamma^2 \int_{n\gamma}^t \int_{n\gamma}^r \left|\nabla^2U(x_s) \nabla U(x_s) - \vec{\Delta}(\nabla U)(x_s)\right|^2\,ds\,dr\right]\right)^{1/2}\\
&\quad  +\sqrt{C\gamma^2 V_c(\bar{x}_{n\gamma})} \\
&\qquad \times\left(\mathbb{E}^{\mathscr{F}_{n\gamma}}\left[\gamma^2 \int_{n\gamma}^t \int_{n\gamma}^r \left| \nabla^2 U(\bar{x}_s) \nabla \tilde{U}_{\gamma}(s, \bar{x}_{n\gamma}) - \vec{\Delta}(\nabla U)(\bar{x}_s)\right|^2\,ds\,dr\right]\right)^{1/2}\\
& \quad +\sqrt{C\gamma^2 V_c(\bar{x}_{n\gamma})} \left(\gamma \int_{n\gamma}^t \int_{n\gamma}^r\mathbb{E}^{\mathscr{F}_{n\gamma}}\left[ \left|\nabla^2 U(x_s)-\nabla^2 U(x_{n\gamma})\right|^2\right]\,ds\,dr\right)^{1/2}\\
& \quad  +\sqrt{C\gamma^2 V_c(\bar{x}_{n\gamma})} \left(\gamma \int_{n\gamma}^t \int_{n\gamma}^r \mathbb{E}^{\mathscr{F}_{n\gamma}}\left[ \left|\nabla^2 U(\bar{x}_s)-\nabla^2 U(\bar{x}_{n\gamma})\right|^2\right]\,ds\,dr\right)^{1/2}\\
& \quad + C\gamma^3(V_c(\bar{x}_{n\gamma})+V_c(x_{n\gamma})) \\
& \leq \sqrt{C\gamma^2 V_c(\bar{x}_{n\gamma})} \left(\gamma \int_{n\gamma}^t \int_{n\gamma}^r\mathbb{E}^{\mathscr{F}_{n\gamma}}\left[(1+|x_s|+|x_{n\gamma}|)^{2\rho-2}|x_s -x_{n\gamma}|^2\right]\,ds\,dr\right)^{1/2}\\
&  +\sqrt{C\gamma^2 V_c(\bar{x}_{n\gamma})} \left(\gamma \int_{n\gamma}^t \int_{n\gamma}^r\mathbb{E}^{\mathscr{F}_{n\gamma}}\left[ (1+|\bar{x}_s|+|\bar{x}_{n\gamma}|)^{2\rho-2}|\bar{x}_s -\bar{x}_{n\gamma}|^2\right]\,ds\,dr\right)^{1/2}\\
& \quad + C\gamma^3(V_c(\bar{x}_{n\gamma})+V_c(x_{n\gamma})).
\end{align*}
Finally by using Cauchy-Schwarz inequality and Lemma \ref{rate2}, one obtains
\begin{align*}
 &\mathbb{E}^{\mathscr{F}_{n\gamma}}\left[\int_{n\gamma}^t(\nabla U(x_r)-\nabla U(\bar{x}_r))\,dr\int_{n\gamma}^tM(\bar{x}_r,\bar{x}_{n\gamma})\,dw_r\right] \\
 &\quad\leq C\gamma^3(V_c(\bar{x}_{n\gamma})+V_c(x_{n\gamma})).
\end{align*}
\end{proof}
\noindent \textbf{Proof of Theorem \ref{thmwasserstein}.}
For $t>0$, consider the coupling
\begin{equation*}
\begin{cases}
 x_t = x_0-\int_0^t\nabla U(x_r)dr +\sqrt{2}w_t, \\
 \bar{x}_{t} = \bar{x}_0 - \int_0^t \nabla\tilde{ U}_{\gamma}(r,\bar{x}_{\floor{r/\gamma}\gamma}) \, dr+ \sqrt{2}w_t,
\end{cases}
\end{equation*}
where $ - \nabla\tilde{ U}_{\gamma}(r,\bar{x}_{\floor{r/\gamma}\gamma})$ is defined in \eqref{eqnintp}. Let $(x_0, \bar{x}_0)$ be distributed according to $\zeta_0$, where $\zeta_0 = \pi \otimes \delta_x$ for all $x \in \mathbb{R}^d$. Define $e_t =x_t-\bar{x}_t$, for all $t \in [n\gamma, (n+1)\gamma)$, $n \in \mathbb{N}$. By It\^o's formula, one obtains, almost surely,
\begin{align*}
|e_t|^2	& = |e_{n\gamma}|^2 -2\int_{n\gamma}^{t} e_s( \nabla U(x_s)-\nabla\tilde{ U}_{\gamma}(s,\bar{x}_{n\gamma}))\,ds.
\end{align*}
Then, taking the expectation and taking the derivative on both sides yield
\begin{align*}
\frac{d}{dt} \mathbb{E}\left[|e_t|^2\right] 	& = -2 \mathbb{E}\left[e_t( \nabla U(x_t)-\nabla\tilde{ U}_{\gamma}(t,\bar{x}_{n\gamma}))\right]\\
														& = 2\mathbb{E}\left[e_t(- (\nabla U(x_t)- \nabla U(\bar{x}_t)))\right]\\
														& \hspace{1em} +2\mathbb{E}\left[e_t(-( \nabla U(\bar{x}_t)- \nabla U(\bar{x}_{n\gamma})- \nabla U_{1,\gamma}(t,\bar{x}_{n\gamma})-\nabla U_{2,\gamma}(t,\bar{x}_{n\gamma})))\right]\\
														& \hspace{1em} +2\mathbb{E}\left[e_t(-(\nabla U(\bar{x}_{n\gamma})-\nabla U_{\gamma}(\bar{x}_{n\gamma})))\right],
\end{align*}
which implies by using \ref{h4} and $|a||b|\leq \varepsilon a^2+(4\varepsilon)^{-1}b^2$, $\varepsilon >0$,
\begin{align}\label{rateeqn1}
\begin{split}
\frac{d}{dt} \mathbb{E}\left[|e_t|^2\right] 	& \leq (2\varepsilon)^{-1}\gamma^3\mathbb{E}\left[|\nabla U(\bar{x}_{n\gamma})|^5\right] -2(m-\varepsilon)\mathbb{E}\left[|e_t|^2\right] \\
														&  +2\mathbb{E}\left[e_t(-( \nabla U(\bar{x}_t)- \nabla U(\bar{x}_{n\gamma})- \nabla U_{1,\gamma}(t,\bar{x}_{n\gamma})-\nabla U_{2,\gamma}(t,\bar{x}_{n\gamma})))\right].
\end{split}
\end{align}
By applying It\^o's formula to $\nabla U(\bar{x}_t)$, and by calculating $\nabla U(\bar{x}_t)- \nabla U(\bar{x}_{n\gamma})- \nabla U_{1,\gamma}(t,\bar{x}_{n\gamma})-\nabla U_{2,\gamma}(t,\bar{x}_{n\gamma})$, one obtains \eqref{rateeqn3}. Substituing \eqref{rateeqn3} into \eqref{rateeqn1} gives
\begin{align*}
&\frac{d}{dt} \mathbb{E}\left[|e_t|^2\right] \\
														& \leq (2\varepsilon)^{-1}\gamma^3\mathbb{E}\left[|\nabla U(\bar{x}_{n\gamma})|^5\right]  -2(m-\varepsilon)\mathbb{E}\left[|e_t|^2\right] \\
														&  \hspace{1em} +2\mathbb{E}\left[|e_t||\int_{n\gamma}^t(\nabla^2 U(\bar{x}_r)-\nabla^2 U(\bar{x}_{n\gamma}))\nabla U_{\gamma}(\bar{x}_{n\gamma})\,dr|\right]\\
														&  \hspace{1em} +2\mathbb{E}\left[|e_t||\int_{n\gamma}^t\nabla^2U(\bar{x}_r)(\nabla U_{1,\gamma}(r,\bar{x}_{n\gamma})+\nabla U_{2,\gamma}(r,\bar{x}_{n\gamma}))\,dr|\right]\\
														&  \hspace{1em} +2\sqrt{2}\mathbb{E}\left[e_t\left(-\int_{n\gamma}^t(\nabla^2 U(\bar{x}_r)-\nabla^2 U(\bar{x}_{n\gamma}))\,dw_r\right)\right]\\
														&  \hspace{1em} +2\mathbb{E}\left[|e_t||\int_{n\gamma}^t( \vec{\Delta}(\nabla U)(\bar{x}_r)- \vec{\Delta}(\nabla U)(\bar{x}_{n\gamma}))\,dr|\right]\\
														&  \hspace{1em} +2\mathbb{E}\left[|e_t|(|\nabla^2U(\bar{x}_{n\gamma})||\nabla U(\bar{x}_{n\gamma})|^2\gamma^2+|\bar{x}_{n\gamma}||\nabla^2U(\bar{x}_{n\gamma})|^2|\nabla U(\bar{x}_{n\gamma})|^2\gamma^2\right.\\
														&\hspace{6em}\left.+\gamma^{3/2}|\bar{x}_{n\gamma}||\vec{\Delta}(\nabla U)(\bar{x}_{n\gamma})|^2+\sqrt{2}\gamma|\nabla^2 U(\bar{x}_{n\gamma})|^2(w_t-w_{n\gamma}))\right].
\end{align*}
By Young's inequality and Cauchy-Schwarz inequality,
\begin{align}\label{eqn:rate}
\frac{d}{dt} \mathbb{E}\left[|e_t|^2\right] 
														& \leq J_1(t)+J_2(t), 
\end{align}
where
\[
J_1(t)	=2\sqrt{2}\mathbb{E}\left[e_t\left(-\int_{n\gamma}^t(\nabla^2 U(\bar{x}_r)-\nabla^2 U(\bar{x}_{n\gamma}))\,dw_r\right)\right],
\]
and
\begin{align*}
&J_2(t) \\
& = (2\varepsilon)^{-1}\gamma^3\mathbb{E}\left[|\nabla U(\bar{x}_{n\gamma})|^5\right]  -2(m-5\varepsilon)\mathbb{E}\left[|e_t|^2\right] \\
														&  \hspace{1em} +(2\varepsilon)^{-1}\gamma\mathbb{E}\left[\int_{n\gamma}^t|(\nabla^2 U(\bar{x}_r)-\nabla^2 U(\bar{x}_{n\gamma}))\nabla U_{\gamma}(\bar{x}_{n\gamma})|^2\,dr\right]\\
														&  \hspace{1em} +(2\varepsilon)^{-1}\gamma\mathbb{E}\left[\int_{n\gamma}^t|\nabla^2U(\bar{x}_r)(\nabla U_{1,\gamma}(r,\bar{x}_{n\gamma})+\nabla U_{2,\gamma}(r,\bar{x}_{n\gamma}))|^2\,dr\right]\\
														&  \hspace{1em} +(2\varepsilon)^{-1}\gamma\mathbb{E}\left[\int_{n\gamma}^t| \vec{\Delta}(\nabla U)(\bar{x}_r)- \vec{\Delta}(\nabla U)(\bar{x}_{n\gamma})|^2\,dr\right]\\
														&  \hspace{1em} +4(2\varepsilon)^{-1}\gamma^3\mathbb{E}\left[|\nabla^2U(\bar{x}_{n\gamma})|^2|\nabla U(\bar{x}_{n\gamma})|^4+|\bar{x}_{n\gamma}|^2|\nabla^2U(\bar{x}_{n\gamma})|^4|\nabla U(\bar{x}_{n\gamma})|^4\right.\\
														&\hspace{15em}\left.+|\bar{x}_{n\gamma}|^2|\vec{\Delta}(\nabla U)(\bar{x}_{n\gamma})|^4+2 |\nabla^2 U(\bar{x}_{n\gamma})|^4\right].
\end{align*}
By taking $\varepsilon = \frac{m}{12}$, and by using the results form \eqref{rateg1} - \eqref{rateg4} in Lemma \ref{rate4terms}, one obtains
\begin{equation}\label{eqn:J2}
J_2(t)	\leq C\gamma^{2+\beta}\mathbb{E}\left[V_c(\bar{x}_{n\gamma})\right] -\frac{7}{6}m\mathbb{E}\left[|e_t|^2\right],
\end{equation}
where $\beta = (0, 1]$. Moreover, one can rewrite $J_1(t)$ as follows
\begin{align*}
\begin{split}
J_1(t)					&=-2\sqrt{2}\mathbb{E}\left[e_t\int_{n\gamma}^t(\nabla^2 U(\bar{x}_r)-\nabla^2 U(\bar{x}_{n\gamma})- M(\bar{x}_r, \bar{x}_{n\gamma}))\,dw_r\right]\\
						& \quad -2\sqrt{2}\mathbb{E}\left[(e_t-e_{n\gamma})\int_{n\gamma}^t  M(\bar{x}_r, \bar{x}_{n\gamma})\,dw_r\right]\\
						&\quad -2\sqrt{2}\mathbb{E}\left[e_{n\gamma}\int_{n\gamma}^t  M(\bar{x}_r, \bar{x}_{n\gamma})\,dw_r\right],
\end{split}
\end{align*}
which implies due to Young's inequality, Lemma \ref{mvt}, \ref{rate2} and the fact that the last term above is zero,
\begin{align*}
\begin{split}
J_1(t)					&\leq 2\varepsilon\mathbb{E}\left[|e_t|^2\right]+C\gamma^{2+\beta} \mathbb{E}\left[V_c(\bar{x}_{n\gamma})\right]\\
						&\hspace{1em} +2\sqrt{2}\mathbb{E}\left[\int_{n\gamma}^t(\nabla U(x_r)-\nabla\tilde{U}_{\gamma}(r,\bar{x}_{n\gamma}))\,dr\int_{n\gamma}^tM(\bar{x}_r, \bar{x}_{n\gamma})\,dw_r\right].
\end{split}
\end{align*}
It can be further rewritten as
\begin{align*}
J_1(t)		&\leq 2\varepsilon\mathbb{E}\left[|e_t|^2\right]+C\gamma^{2+\beta} \mathbb{E}\left[V_c(\bar{x}_{n\gamma})\right]\\
			&\hspace{1em}+2\sqrt{2}\mathbb{E}\left[\int_{n\gamma}^t(\nabla U(x_r)-\nabla U(\bar{x}_r))\,dr\int_{n\gamma}^tM(\bar{x}_r, \bar{x}_{n\gamma})\,dw_r\right]\\
			& \hspace{1em}+ 2\sqrt{2}\mathbb{E}\left[\int_{n\gamma}^t(\nabla U(\bar{x}_r)- \nabla U(\bar{x}_{n\gamma})- \nabla U_{1,\gamma}(r,\bar{x}_{n\gamma})-\nabla U_{2,\gamma}(r,\bar{x}_{n\gamma}))\,dr\right.\\
			&\hspace{8em}\left.\times\int_{n\gamma}^tM(\bar{x}_r, \bar{x}_{n\gamma})\,dw_r\right]\\
			& \hspace{1em}+2\sqrt{2}\mathbb{E}\left[\int_{n\gamma}^t(\nabla U(\bar{x}_{n\gamma})-\nabla U_{\gamma}(\bar{x}_{n\gamma}))\,dr\int_{n\gamma}^tM(\bar{x}_r, \bar{x}_{n\gamma})\,dw_r\right],
\end{align*}
which, by using Cauchy-Schwarz inequality, Remark \ref{poly}, Lemma \ref{problematicterm}, \ref{rate4terms} and \ref{rateM}, yields
\begin{align}\label{eqn:J1}
J_1(t)		\leq 2\varepsilon\mathbb{E}\left[|e_t|^2\right]+C\gamma^{2+\beta}\mathbb{E}\left[V_c(x_{n\gamma})+V_c(\bar{x}_{n\gamma})\right]
\end{align}
Substituting \eqref{eqn:J1} and \eqref{eqn:J2} into \eqref{eqn:rate} with $\varepsilon = \frac{m}{12}$, one obtains the following result,
\begin{align*}
\frac{d}{dt} \mathbb{E}\left[|e_t|^2\right] 	& \leq -m\mathbb{E}\left[|e_t|^2\right]+ C\gamma^{2+\beta}\mathbb{E}\left[V_c(x_{n\gamma})+V_c(\bar{x}_{n\gamma})\right].
\end{align*}
The application of Gronwall's lemma yields
\begin{align*}
 \mathbb{E}\left[|e_t|^2\right] 	 \leq e^{-m(t-n\gamma)}\mathbb{E}\left[|e_{n\gamma}|^2\right]+ C\gamma^{3+\beta}\mathbb{E}\left[V_c(x_{n\gamma})+V_c(\bar{x}_{n\gamma})\right].
\end{align*}
Finally, by induction, Proposition \ref{sdemomentbound} and \ref{momentbound}, one obtains
\begin{align*}
\begin{split}
 \mathbb{E}\left[|e_{(n+1)\gamma}|^2\right]	& \leq e^{-m\gamma (n+1)}\mathbb{E}\left[|e_0|^2\right]\\
 															&\quad +C\gamma^{3+\beta}\sum_{k=0}^n\mathbb{E}\left[V_c(\bar{x}_{k\gamma})+V_c(x_{k\gamma})\right]e^{-m\gamma (n-k)}\\
 															& \leq e^{-m\gamma (n+1)}\mathbb{E}\left[|x_0-\bar{x}_0|^2\right]+\frac{3bC}{7c^2m}e^{(\frac{7}{3}c^2+m)\gamma}\gamma^{2+\beta}\\
 															&\quad+\frac{C}{m}\gamma^{2+\beta}(\mathbb{E}\left[V_c(x_0)\right]+b_a)e^{m\gamma}\\
 															&\quad +C\gamma^{3+\beta}\mathbb{E}\left[V_c(\bar{x}_0)\right]\sum_{k=0}^ne^{-\frac{7}{3}c^2\gamma k-m\gamma (n-k)},
\end{split}
\end{align*}
where the last inequality holds by using $1-e^{-m \gamma } \geq m \gamma e^{-m \gamma}$, and this indicates (see Appendix \ref{wassersteinind} for a detailed proof)
\begin{equation}\label{wassersteinindeq}
 \mathbb{E}\left[|e_{(n+1)\gamma}|^2\right]	 \leq  e^{-m\gamma (n+1)}\mathbb{E}\left[|x_0-\bar{x}_0|^2\right]+C \gamma^{2+\beta},
\end{equation}
Note that $(x_0, \bar{x}_0)$ is distributed according to $\zeta_0$, then \eqref{wasserstein1} can be obtained by using Theorem 1 in \cite{DM16} and the triangle inequality.

\subsection{Proof of Theorem \ref{thmtv}}
By applying the following lemma, one can show that without using \ref{h4}, the rate of convergence in total variation norm is of order 1, which is properly stated in Theorem \ref{thmtv}.
\begin{lemma}\label{lemmatv} Asuume \ref{h1} and \ref{h3} are satisfied. Let $p \in \mathbb{N}$ and $\nu_0$ be a probability measure on $(\mathbb{R}^d, \mathcal{B}(\mathbb{R}^d))$. There exists $C>0$ such that for all $\gamma \in (0,1)$
\[
\mathrm{KL}(\nu_0R_{\gamma}^p | \nu_0P_{p\gamma}) \leq C \gamma^3 \int_{\mathbb{R}^d} \sum_{i=0}^{p-1}\left(\int_{\mathbb{R}^d}V_c(z)R_{\gamma}^i(y,dz)\right)\nu_0(dy).
\]
\end{lemma}
\begin{proof}
Denote by $\mu_p^y$ and $\bar{\mu}_p^y$ the laws on $\mathcal{C}([0, p\gamma], \mathbb{R}^d)$ of the SDE \eqref{eqn:sde} and of the linear interpolation \eqref{eqnintp} of the scheme both started at $y \in \mathbb{R}^d$. Denote by $(\mathscr{F}_t)_{t\geq 0}$ the filtration associated with $(w_t)_{t\geq 0}$, and by $(x_t, \bar{x}_t)_{t \geq 0}$ the unique strong solution of
\begin{equation}\label{eqnsol}
\begin{cases}
 dx_t = -\nabla U(x_t)dt +\sqrt{2}dw_t, \\
 d\bar{x}_{t} = - \nabla\tilde{ U}_{\gamma}(t,\bar{x}_{\floor{t/\gamma}\gamma}) \, dt+ \sqrt{2}w_t,
\end{cases}
\end{equation}
where $ - \nabla\tilde{ U}_{\gamma}(t,\bar{x}_{\floor{t/\gamma}\gamma})$ is defined in \eqref{eqnintp}. Then, by taking into consideration Definition 7 concerning diffusion type processes and Lemma 4.9 which refers to their representations in section 4.2 from \cite{lipshiry}, Theorem 7.19 in \cite{lipshiry} can be applied to obtain the Radon-Nikodym derivative of $\mu_p^y$ w.r.t. $\bar{\mu}_p^y$, i.e.
\begin{align} \label{rdderivative}
\begin{split}
\frac{d\mu_p^y}{d\bar{\mu}_p^y}((\bar{x}_t)_{t \in [0, p\gamma]})	& = \exp\left(\frac{1}{2}\int_0^{p\gamma}(-\nabla U(\bar{x}_s)+\nabla\tilde{ U}_{\gamma}(s, \bar{x}_{\floor{s/\gamma}\gamma}))d\bar{x}_s\right. \\
																	& \quad \left. -\frac{1}{4}\int_0^{p\gamma}\left(|\nabla U(\bar{x}_s)|^2 -|\nabla\tilde{ U}_{\gamma}(s, \bar{x}_{\floor{s/\gamma}\gamma})|^2\right)\,ds\right).
\end{split}
\end{align}
Note that the assumptions of Theorem 7.19 in \cite{lipshiry} are satisfied due to proposition \ref{sdemomentbound} and \ref{momentbound}. By using \eqref{rdderivative}, one obtains
\begin{align*} 
\mathrm{KL}(\bar{\mu}_p^y|\mu_p^y) 	& = \mathbb{E}_y \left(-\log\left(\frac{d\mu_p^y}{d\bar{\mu}_p^y}((\bar{x}_t)_{t \in [0, p\gamma]}) \right)\right)\\
													& = \frac{1}{4} \int_0^{p\gamma}\mathbb{E}_y\left(\left|\nabla U(\bar{x}_s)-\nabla\tilde{ U}_{\gamma}(s,\bar{x}_{\floor{s/\gamma}\gamma})\right|^2\right)\,ds\\
													&= \frac{1}{4} \sum_{i=0}^{p-1}\int_{i\gamma}^{(i+1)\gamma}\mathbb{E}_y\left(\left|\nabla U(\bar{x}_s)-\nabla\tilde{ U}_{\gamma}(s,\bar{x}_{i\gamma})\right|^2\right)\,ds\\
													& \leq \frac{1}{2} \sum_{i=0}^{p-1}\int_{i\gamma}^{(i+1)\gamma}\mathbb{E}_y\left(\mathbb{E}^{\mathscr{F}_{i\gamma}}\left(\left|\nabla U(\bar{x}_s)-\nabla U(\bar{x}_{i\gamma}) \right.\right.\right.\\
													&\hspace{6em} \left.\left.\left. -\nabla U_{1,\gamma}(s,\bar{x}_{i\gamma}) -\nabla U_{2,\gamma}(s,\bar{x}_{i\gamma})\right|^2\right)\right)\,ds\\
													&\hspace{1em}+\frac{1}{2} \sum_{i=0}^{p-1}\int_{i\gamma}^{(i+1)\gamma}\mathbb{E}_y\left(\mathbb{E}^{\mathscr{F}_{i\gamma}}\left(\left|\nabla U(\bar{x}_{i\gamma})-\nabla U_{\gamma}(\bar{x}_{i\gamma})\right|^2\right)\right)\,ds\\
													&\leq C\gamma^3 \sum_{i=0}^{p-1}\mathbb{E}_y\left(V_c(\bar{x}_{i\gamma})\right), 
\end{align*}
where the last inequality holds due to Lemma \ref{rate4terms}. Then, by Theorem 4.1 in \cite{kullback}, it follows that
\[
\mathrm{KL}(\delta_y R_{\gamma}^p|\delta_y P_{p\gamma})  \leq \mathrm{KL}(\bar{\mu}_p^y|\mu_p^y) \leq C\gamma^3 \sum_{i=0}^{p-1}\mathbb{E}_y\left(V_c(\bar{x}_{i\gamma})\right).
\]
Finally, applying the tower property yields the desired result,
\begin{align*}
\mathrm{KL}(\nu_0R_{\gamma}^p | \nu_0P_{p\gamma})& \leq C\gamma^3 \sum_{i=0}^{p-1}\mathbb{E}\left(\mathbb{E}_y\left(V_c(\bar{x}_{i\gamma})\right) \right)\\
& = C \gamma^3 \int_{\mathbb{R}^d} \sum_{i=0}^{p-1}\left(\int_{\mathbb{R}^d}V_c(z)R_{\gamma}^i(y,dz)\right)\nu_0(dy).
\end{align*}
\end{proof}
\noindent \textbf{Proof of Theorem \ref{thmtv}.} The proof follows along the same lines as the proof of Theorem 4 in \cite{tula}, but for the completeness, the details are given below.

By Proposition \ref{sdemomentbound}, for all $n \in \mathbb{N}$ and $x \in \mathbb{R}^d$, we have
\begin{align*}
\|\delta_x R_{\gamma}^n - \pi\|_{V_c^{1/2}}
		& \leq \|\delta_x P_{n\gamma}-\pi\|_{V_c^{1/2}}+\|\delta_x R_{\gamma}^n-\delta_x P_{n\gamma}\|_{V_c^{1/2}}\\
		&  \leq C_{c/2}\rho^{n\gamma}_{c/2}V_{c}^{1/2}(x)+\|\delta_x R_{\gamma}^n-\delta_x P_{n\gamma}\|_{V_c^{1/2}}.
\end{align*}
Denote by $k_{\gamma}=\ceil{\gamma^{-1}}$, and by $q_{\gamma}$, $r_{\gamma}$ the quotient and the remainder of the Euclidian division of $n$ by $k_{\gamma}$, i.e. $n = q_{\gamma}k_{\gamma}+r_{\gamma}$. Then, 
\[
\|\delta_x R_{\gamma}^n-\delta_x P_{n\gamma}\|_{V_c^{1/2}} \leq I_1 +I_2,
\]
where 
\begin{align*}
I_1 	&= \|\delta_xR_{\gamma}^{q_{\gamma}k_{\gamma}}P_{r_{\gamma}\gamma}-\delta_xR_{\gamma}^n\|_{V_c^{1/2}}\\
I_2 	&= \sum_{i=1}^{q_{\gamma}}\|\delta_xR_{\gamma}^{(i-1)k_{\gamma}}P_{(n-(i-1)k_{\gamma})\gamma}-\delta_xR_{\gamma}^{ik_{\gamma}}P_{(n-ik_{\gamma})\gamma}\|_{V_c^{1/2}}\\
		& \leq \sum_{i=1}^{q_{\gamma}}C_{c/2}\rho^{(n-ik_{\gamma})\gamma}_{c/2}\|\delta_xR_{\gamma}^{(i-1)k_{\gamma}}P_{k_{\gamma}\gamma}-\delta_xR_{\gamma}^{ik_{\gamma}}\|_{V_c^{1/2}}
\end{align*}
By applying Lemma 24 in \cite{DM17} to $I_1$, we have
\begin{align}\label{tveqn1}
\|\delta_xR_{\gamma}^{q_{\gamma}k_{\gamma}}P_{r_{\gamma}\gamma}-\delta_xR_{\gamma}^n\|^2_{V_c^{1/2}} &\leq 2 \left(\delta_xR_{\gamma}^{q_{\gamma}k_{\gamma}}P_{r_{\gamma}\gamma}(V_c)+\delta_xR_{\gamma}^n(V_c)\right)\nonumber \\
&\quad\times \mathrm{KL}(\delta_xR_{\gamma}^n|\delta_xR_{\gamma}^{q_{\gamma}k_{\gamma}}P_{r_{\gamma}\gamma}).
\end{align}
Then, by Proposition \ref{momentbound} and Lemma \ref{lemmatv}, one obtains
\begin{align}\label{tveqn2}
\begin{split}
\mathrm{KL}(\delta_xR_{\gamma}^n|\delta_xR_{\gamma}^{q_{\gamma}k_{\gamma}}P_{r_{\gamma}\gamma}) 
	&\leq  C\gamma^3\sum_{j=0}^{r_{\gamma}-1}\int_{\mathbb{R}^d}V_c(z)\delta_xR_{\gamma}^{q_{\gamma}k_{\gamma}+j}(dz)\\
	& \leq  C\gamma^3(1+\gamma^{-1})\left(e^{-\frac{7}{3}c^2q_{\gamma}k_{\gamma}\gamma}V_c(x)+\frac{3b}{7c^2}e^{\frac{7}{3}c^2\gamma}\right),
\end{split}
\end{align}
where the last inequality holds since $r_{\gamma} \leq k_{\gamma} \leq 1+\gamma^{-1}$. Furthermore, by Proposition \ref{sdemomentbound} and Proposition \ref{momentbound},
\begin{equation}\label{tveqn3}
\delta_xR_{\gamma}^{q_{\gamma}k_{\gamma}}P_{r_{\gamma}\gamma}(V_c)+\delta_xR_{\gamma}^n(V_c) \leq 2\left(e^{-\frac{7}{3}c^2q_{\gamma}k_{\gamma}\gamma}V_c(x)+\frac{3b}{7c^2}e^{\frac{7}{3}c^2\gamma}+b_c\right).
\end{equation}
Substituting \eqref{tveqn2} and \eqref{tveqn3} into \eqref{tveqn1} yields
\begin{align*}
I_1& \leq 2C^{1/2}\gamma^{3/2}(1+\gamma^{-1})^{1/2}\left(e^{-\frac{7}{3}c^2q_{\gamma}k_{\gamma}\gamma}V_c(x)+\frac{3b}{7c^2}e^{\frac{7}{3}c^2\gamma}+b_c\right)\\
& \leq C (\lambda^{n\gamma}V_c(x)+\gamma),
\end{align*}
where $\lambda \in (0,1)$. By using similar arguments to $I_2$, one obtains \eqref{tvthm1}.

\section{Lipschitz case}\label{lipcase}
In the context of a Lipschitz gradient, assume \ref{h4} - \ref{h7} hold. Then, by \ref{h5} and \ref{h6}, one obtains, for any $x, y \in \mathbb{R}^d$
\begin{equation}\label{usefulineq}
|\nabla^2 U(x)y| \leq L_1 |y|, \quad |\vec{\Delta}(\nabla U(x))| \leq d L_2.
\end{equation}
One also notice that by \cite[Theorem~2.1.12]{nesterov}, under \ref{h4} and \ref{h5}, for all $x, y \in \mathbb{R}^{d}$,
\begin{equation}\label{h4-alt}
(x-y)\left(\nabla U(x)-\nabla U(y)\right) \geq \tilde{m} | x - y|^2 +\frac{1}{m+L_1}|\nabla U(x)-\nabla U(y)|^2,
\end{equation}
where we have set
\begin{equation}
\label{eq:definition-tilde-m}
\tilde{m} = \frac{mL_1}{m+L_1} \,.
\end{equation}
The linear interpolation of the algorithm \eqref{eqnschemelip} becomes
\begin{align}\label{lipscheme}
\begin{split} 
&\tilde{x}_{t} = \tilde{x}_0-\int_0^{t} \nabla\tilde{ U}(s,\tilde{x}_{\floor{s/\gamma}\gamma}) \, ds+ \sqrt{2}w_t,
\end{split}
\end{align}
for all $t \geq 0$, where
\[
\nabla\tilde{ U} (s,\tilde{x}_{\floor{s/\gamma}\gamma}) =  \nabla U (\tilde{x}_{\floor{s/\gamma}\gamma})+\nabla U_1(s, \tilde{x}_{\floor{s/\gamma}\gamma}) +\nabla U_2(s,\tilde{x}_{\floor{s/\gamma}\gamma}),
\]
with
\begin{align*} \label{u1comparison}
&\nabla U_1(s,\tilde{x}_{\floor{s/\gamma}\gamma}) \\
&=- \int_{\floor{s/\gamma}\gamma}^s \left( \nabla^2 U(\tilde{x}_{\floor{s/\gamma}\gamma})\nabla U(\tilde{x}_{\floor{s/\gamma}\gamma})  -\vec{\Delta}(\nabla U)(\tilde{x}_{\floor{s/\gamma}\gamma})\right)\, dr, 
\end{align*}
and
\[
\nabla U_2(s,\tilde{x}_{\floor{s/\gamma}\gamma})=\sqrt{2} \int_{\floor{s/\gamma}\gamma}^s \nabla^2 U (\tilde{x}_{\floor{s/\gamma}\gamma})\, dw_r.
\]
One notes that for any $n \in \mathbb{N}$, $\tilde{X}_n = \tilde{x}_{n \gamma}$.
\subsection{Moment bounds}
\begin{prop}\label{lip2ndmoment} Assume \ref{h4} - \ref{h7} are satisfied. Let $x^{\ast}$ be the unique minimizer of $U$. Then, for all $x \in \mathbb{R}^d$,  $\gamma \in \left(0,\frac{1}{\tilde{m}} \wedge \frac{8\tilde{m}^2}{m( 2L_1^2+7\tilde{m}L_1)}\right)$ and $n \in \mathbb{N}$,
\[
\mathbb{E}^{\mathscr{F}_{0}}|\tilde{x}_{(n+1)\gamma}-x^{\ast}|^2 \leq (1- \tilde{m}\gamma)^{n+1}|\tilde{x}_0 - x^{\ast}|^2+\frac{q_1}{\tilde{m}},
\]
where $q_1 = \left(\frac{L_2^2}{2\tilde{m}}+\frac{ 3L_2^2}{2} \right)d^2 +(4L_1^2 +4)d$ and $\tilde{m}$ is given in \eqref{eq:definition-tilde-m}.
\end{prop}
\begin{proof}
Denote by 
\[
\Delta_n = \tilde{x}_{n\gamma}-x^{\ast}- \nabla U(\tilde{x}_{n\gamma})\gamma+\frac{\gamma^2}{2}\left(\nabla^2U(\tilde{x}_{n\gamma})\nabla U(\tilde{x}_{n\gamma})-\vec{\Delta}(\nabla U)(\tilde{x}_{n\gamma})\right),
\] 
where $x^{\ast}$ is the unique minimizer of $U$ and one calculates
\begin{align}
\begin{split}\label{2ndmoment}
&\mathbb{E}^{\mathscr{F}_{n\gamma}}|\tilde{x}_{(n+1)\gamma}-x^{\ast}|^2	\\
&= \mathbb{E}^{\mathscr{F}_{n\gamma}}\left|\Delta_n -\sqrt{2}\int_{n\gamma}^{(n+1)\gamma} \int_{n\gamma}^r \nabla^2U( \tilde{x}_{n\gamma})\,dw_{s}\,dr +\sqrt{2}\int_{n\gamma}^{(n+1)\gamma}\,dw_r\right|^2 \\
																							& =|\Delta_n|^2	+4\mathbb{E}^{\mathscr{F}_{n\gamma}}\left|\int_{n\gamma}^{(n+1)\gamma} \int_{n\gamma}^r \nabla^2U( \tilde{x}_{n\gamma})\,dw_{s}\,dr \right|^2+4d\gamma  \\
																							& \leq |\Delta_n|^2+4\gamma^3  L_1^2d+4\gamma d,
\end{split}
\end{align}
where the last inequality holds due to \eqref{usefulineq}. Then, by using \eqref{usefulineq}, \eqref{h4-alt} and the strong convexity condition of $U$, one obtains, for $\gamma \in \left(0,\frac{1}{\tilde{m}} \wedge \frac{8\tilde{m}^2}{m( 2L_1^2+7\tilde{m}L_1)}\right)$ and $n \in \mathbb{N}$, 
\begin{align*}
&|\Delta_n|^2	\\
& = |\tilde{x}_{n\gamma}-x^{\ast}|^2+\left|- \nabla U(\tilde{x}_{n\gamma})\gamma+\frac{\gamma^2}{2}\left(\nabla^2U(\tilde{x}_{n\gamma})\nabla U(\tilde{x}_{n\gamma})-\vec{\Delta}(\nabla U)(\tilde{x}_{n\gamma})\right)\right|^2 \\
&\quad +2(\tilde{x}_{n\gamma}-x^{\ast})\left(- \nabla U(\tilde{x}_{n\gamma})\gamma+\frac{\gamma^2}{2}\left(\nabla^2U(\tilde{x}_{n\gamma})\nabla U(\tilde{x}_{n\gamma})-\vec{\Delta}(\nabla U)(\tilde{x}_{n\gamma})\right)\right) \nonumber\\
		& \leq (1- 2\tilde{m}\gamma)|\tilde{x}_{n\gamma}-x^{\ast}|^2 - \frac{2\gamma}{m+L_1}|\nabla U(\tilde{x}_{n\gamma}) - \nabla U(x^{\ast})|^2\nonumber \\
		&\quad + \tilde{m}\gamma|\tilde{x}_{n\gamma}-x^{\ast}|^2 +\frac{\gamma^3}{4\tilde{m}}\left|\nabla^2U(\tilde{x}_{n\gamma})\nabla U(\tilde{x}_{n\gamma})-\vec{\Delta}(\nabla U)(\tilde{x}_{n\gamma})\right|^2 \nonumber\\
		&\quad +\gamma^2 |\nabla U(\tilde{x}_{n\gamma}) - \nabla U(x^{\ast})|^2 +\frac{\gamma^4}{4}\left|\nabla^2U(\tilde{x}_{n\gamma})\nabla U(\tilde{x}_{n\gamma})-\vec{\Delta}(\nabla U)(\tilde{x}_{n\gamma})\right|^2\\
		&\quad+\gamma^3\nabla U(\tilde{x}_{n\gamma}) \vec{\Delta}(\nabla U)(\tilde{x}_{n\gamma}),
\end{align*}
which by using $(a+b)^2 \leq 2a^2 +2b^2$ and $2ab \leq a^2+b^2$ for $a, b \geq 0$ yield
\begin{align}\label{deltabd}	
|\Delta_n|^2		& \leq  (1- \tilde{m}\gamma)|\tilde{x}_{n\gamma}-x^{\ast}|^2 \nonumber\\
			&\quad + \left(-\frac{2\gamma}{m+L_1}+\frac{5\gamma^2}{4}+\frac{\gamma^3 L_1^2}{2\tilde{m}}+\frac{\gamma^4L_1^2}{2}\right)|\nabla U(\tilde{x}_{n\gamma}) - \nabla U(x^{\ast})|^2\nonumber\\
		&\quad +\frac{\gamma^3L_2^2}{2\tilde{m}}d^2+\frac{\gamma^4 L_2^2}{2}d^2+\gamma^4L_2^2d^2 \nonumber\\
		& \leq (1- \tilde{m}\gamma)|\tilde{x}_{n\gamma}-x^{\ast}|^2  +\gamma^3\left(\frac{L_2^2}{2\tilde{m}}+\frac{3 L_2^2}{2}\right)d^2,
\end{align}
where $\tilde{m}$ is defined in \eqref{eq:definition-tilde-m}. Substituting the above upper bound into \eqref{2ndmoment} yields
\begin{align*}
\mathbb{E}^{\mathscr{F}_{n\gamma}}|\tilde{x}_{(n+1)\gamma}-x^{\ast}|^2	 \leq (1- \tilde{m}\gamma)|\tilde{x}_{n\gamma}-x^{\ast}|^2  +\gamma q_1,
\end{align*}
where $q_1 = \left(\frac{L_2^2}{2\tilde{m}}+\frac{3 L_2^2}{2}\right)d^2+(4  L_1^2+4)d$, and the result can be obtained by induction.
\end{proof}
\begin{prop} \label{lip4thmoment} Assume \ref{h4} - \ref{h7} are satisfied. Let $x^{\ast}$ be the unique minimizer of $U$. Then, for all $x \in \mathbb{R}^d$, $\gamma \in \left(0,\frac{1}{\tilde{m}} \wedge  \frac{8\tilde{m}^2}{m( 2L_1^2+7\tilde{m}L_1)}\right)$,
\[
\mathbb{E}^{\mathscr{F}_{0}}|\tilde{x}_{(n+1)\gamma}-x^{\ast}|^4 \leq \left(1-\frac{\tilde{m}\gamma}{8}\right)^{n+1}|\tilde{x}_0 - x^{\ast}|^4+\frac{8q_2}{\tilde{m}},
\]
where $q_2 = \left(2+\frac{8}{\tilde{m}\gamma}\right)\left(\frac{L_2^2}{2\tilde{m}}+\frac{3 L_2^2}{2}\right)^2d^4+32\gamma\left(1+\frac{42}{\tilde{m}}\right)(L_1^4+3) d^2$ and $\tilde{m}$ is given in \eqref{eq:definition-tilde-m}.
\end{prop}
\begin{proof}  Denote by 
\[
\Delta_n = \tilde{x}_{n\gamma}-x^{\ast}- \nabla U(\tilde{x}_{n\gamma})\gamma+\frac{\gamma^2}{2}\left(\nabla^2U(\tilde{x}_{n\gamma})\nabla U(\tilde{x}_{n\gamma})-\vec{\Delta}(\nabla U)(\tilde{x}_{n\gamma})\right),
\] 
and 
\[
\tilde{\Delta}_{n+1} = -\sqrt{2}\int_{n\gamma}^{(n+1)\gamma} \int_{n\gamma}^r \nabla^2U( \tilde{x}_{n\gamma})\,dw_{s}\,dr +\sqrt{2}\int_{n\gamma}^{(n+1)\gamma}\,dw_r.
\] 
One obtains by using Jensen's inequality
\begin{align}
&\mathbb{E}^{\mathscr{F}_{n\gamma}}|\tilde{x}_{(n+1)\gamma}-x^{\ast}|^4	\\
&= \mathbb{E}^{\mathscr{F}_{n\gamma}}\left|\Delta_n +\tilde{\Delta}_{n+1}\right|^4 \nonumber\\
&= \mathbb{E}^{\mathscr{F}_{n\gamma}}\left(|\Delta_n|^2	+2\Delta_n\tilde{\Delta}_{n+1}+|\tilde{\Delta}_{n+1}|^2\right)^2\nonumber \\
&  = \mathbb{E}^{\mathscr{F}_{n\gamma}}\left(|\Delta_n|^4	+4\Delta_n\tilde{\Delta}_{n+1}|\Delta_n|^2+2|\Delta_n|^2|\tilde{\Delta}_{n+1}|^2+4|\Delta_n\tilde{\Delta}_{n+1}|^2\right.\\
&\hspace{15em}\left.+4\Delta_n\tilde{\Delta}_{n+1}|\tilde{\Delta}_{n+1}|^2+|\tilde{\Delta}_{n+1}|^4\right)\nonumber \\
																							& \leq |\Delta_n|^4	+6|\Delta_n|^2\mathbb{E}^{\mathscr{F}_{n\gamma}}|\tilde{\Delta}_{n+1}|^2+\mathbb{E}^{\mathscr{F}_{n\gamma}}|\tilde{\Delta}_{n+1}|^4+4|\Delta_n|\mathbb{E}^{\mathscr{F}_{n\gamma}}|\tilde{\Delta}_{n+1}|^3\nonumber\\
																							& \leq \left(1+\frac{\tilde{m}\gamma}{2}\right)|\Delta_n|^4+\frac{36}{\tilde{m}\gamma}\left(\mathbb{E}^{\mathscr{F}_{n\gamma}}|\tilde{\Delta}_{n+1}|^2\right)^2 +\left(1+\frac{6}{\tilde{m}\gamma}\right)\mathbb{E}^{\mathscr{F}_{n\gamma}}|\tilde{\Delta}_{n+1}|^4\nonumber\\
																							& \leq  \left(1+\frac{\tilde{m}\gamma}{2}\right)|\Delta_n|^4 +32\gamma\left(1+\frac{42}{\tilde{m}}\right)(L_1^4+3) d^2.
\end{align}
Then, by using \eqref{deltabd} and the inequality $(a+b)^2 \leq (1+\epsilon)a^2+(1+\epsilon^{-1})b^2$, for any $a, b \geq 0$, $\epsilon >0$, one obtains, for $\gamma \in \left(0,\frac{1}{\tilde{m}} \wedge  \frac{8\tilde{m}^2}{m( 2L_1^2+7\tilde{m}L_1)}\right)$,
\begin{align*}
&\mathbb{E}^{\mathscr{F}_{n\gamma}}|\tilde{x}_{(n+1)\gamma}-x^{\ast}|^4 \\
&\leq \left(1+\frac{\tilde{m}\gamma}{2}\right)\left((1- \tilde{m}\gamma)|\tilde{x}_{n\gamma}-x^{\ast}|^2  +\gamma^3\left(\frac{L_2^2}{2\tilde{m}}+\frac{3 L_2^2}{2}\right)d^2\right)^2 \\
&\quad+32\gamma\left(1+\frac{42}{\tilde{m}}\right)(L_1^4+3) d^2\\
		& \leq \left(1+\frac{\tilde{m}\gamma}{2}\right)\left(1+\frac{\tilde{m}\gamma}{4}\right)(1- \tilde{m}\gamma)|\tilde{x}_{n\gamma}-x^{\ast}|^4\\
		&\quad  +\left(1+\frac{\tilde{m}\gamma}{2}\right)\left(1+\frac{4}{\tilde{m}\gamma}\right)\gamma^6\left(\frac{L_2^2}{2\tilde{m}}+\frac{3 L_2^2}{2}\right)^2d^4\\
		& \quad+32\gamma\left(1+\frac{42}{\tilde{m}}\right)(L_1^4+3) d^2\\
		& \leq \left(1-\frac{\tilde{m}\gamma}{8}\right)|\tilde{x}_{n\gamma}-x^{\ast}|^4 +\gamma^6\left(2+\frac{8}{\tilde{m}\gamma}\right)\left(\frac{L_2^2}{2\tilde{m}}+\frac{3 L_2^2}{2}\right)^2d^4\\
		&\quad+32\gamma\left(1+\frac{42}{\tilde{m}}\right)(L_1^4+3) d^2,
\end{align*}
which implies
\begin{equation}\label{4thmoment}
\mathbb{E}^{\mathscr{F}_{n\gamma}}|\tilde{x}_{(n+1)\gamma}-x^{\ast}|^4 \leq \left(1-\frac{\tilde{m}\gamma}{8}\right)|\tilde{x}_{n\gamma}-x^{\ast}|^4 +\gamma q_2,
\end{equation}
where $q_2 = \left(2+\frac{8}{\tilde{m}}\right)\left(\frac{L_2^2}{2\tilde{m}}+\frac{3 L_2^2}{2}\right)^2d^4+32 \left(1+\frac{42}{\tilde{m}}\right)(L_1^4+3) d^2$. The desired result follows by induction.
\end{proof}
\subsection{Proof of Theorem \ref{thmwassersteinlip}}
The explicit constants for the second and the fourth moments are obtained, then by using the following lemmas, one can show the rate of convergence in Wasserstein distance.
\begin{lemma}\label{liprate1}
Assume \ref{h4} - \ref{h7} are satisfied. Let $\gamma \in \left(0,\frac{1}{\tilde{m}} \wedge  \frac{8\tilde{m}^2}{m( 2L_1^2+7\tilde{m}L_1)}\right)$. Then, for all $n \in \mathbb{N}$, and $t \in [n\gamma, (n+1)\gamma)$,
\[
\mathbb{E}^{\mathscr{F}_{n\gamma}}\left[|\nabla U_1(t,\tilde{x}_{n\gamma})|^2\right] \leq 2\gamma^2(L_1^4|\tilde{x}_{n\gamma}-x^{\ast}|^2+d^2L_2^2),
\]
\[
\mathbb{E}^{\mathscr{F}_{n\gamma}}\left[|\nabla U_1(t,\tilde{x}_{n\gamma})|^4\right] \leq 8\gamma^4(L_1^8|\tilde{x}_{n\gamma}-x^{\ast}|^4+d^4L_2^4), 
\]
\[
\mathbb{E}^{\mathscr{F}_{n\gamma}}\left[|\nabla U_2(t,\tilde{x}_{n\gamma})|^2 \right] \leq  2\gamma d L_1^2,\quad
\mathbb{E}^{\mathscr{F}_{n\gamma}}\left[|\nabla U_2(t,\tilde{x}_{n\gamma})|^4 \right] \leq 12L_1^4d^2\gamma^2.
\]
\end{lemma}
\begin{proof}
The proof is straightforward by using \eqref{usefulineq}.
\end{proof}

\begin{lemma} \label{liprate2}
Assume \ref{h4} - \ref{h7} are satisfied. Let $\gamma \in \left(0,\frac{1}{\tilde{m}} \wedge  \frac{8\tilde{m}^2}{m( 2L_1^2+7\tilde{m}L_1)}\right)$. Then, for all $n \in \mathbb{N}$, and $t \in [n\gamma, (n+1)\gamma)$,
\[
\mathbb{E}^{\mathscr{F}_{n\gamma}}\left[|\tilde{x}_t -\tilde{x}_{n\gamma} |^2\right] \leq \gamma(c_1 |\tilde{x}_{n\gamma}-x^{\ast}|^2+c_2),
\]
where  $c_1= \frac{5L_1^2}{4} +\frac{L_1^4}{2}$ and $c_2 = \frac{3 L_2^2}{2}d^2 +4L_1^2 d +4 d$, 
\[
\mathbb{E}^{\mathscr{F}_{n\gamma}}\left[|\tilde{x}_t -\tilde{x}_{n\gamma} |^4\right] \leq \gamma^2(c_3|\tilde{x}_{n\gamma}-x^{\ast}|^4 +c_4),
\]
where  $c_3=  9 \left(\frac{25}{16}L_1^4 +\frac{L_1^8}{4}\right)$ and $c_4 = \frac{81L_2^4}{4}d^4+416(L_1^4 +3)d^2$, and
\[
\mathbb{E}^{\mathscr{F}_{n\gamma}}\left[|x_t -x_{n\gamma} |^2\right] \leq 2\gamma^2 L_1^2|x_{n\gamma}-x^{\ast}|^2+4\gamma^3L_1^2d+4\gamma d.
\]
\end{lemma}
\begin{proof} One observes that
\begin{align}\label{auxineq}
&\mathbb{E}^{\mathscr{F}_{n\gamma}}|\tilde{x}_t-\tilde{x}_{n\gamma}|^2 \nonumber\\
	& =\left|- \nabla U(\tilde{x}_{n\gamma})(t-n\gamma)+\frac{(t-n\gamma)^2}{2}\left(\nabla^2U(\tilde{x}_{n\gamma})\nabla U(\tilde{x}_{n\gamma})-\vec{\Delta}(\nabla U)(\tilde{x}_{n\gamma})\right)\right|^2\nonumber\\
	& \hspace{1em} +4\mathbb{E}^{\mathscr{F}_{n\gamma}}\left|\int_{n\gamma}^t \int_{n\gamma}^r \nabla^2U( \tilde{x}_{n\gamma})\,dw_{s}\,dr \right|^2
	+4d(t-n\gamma ) \nonumber\\ 
	 &\leq |\nabla U(\tilde{x}_{n\gamma}) - \nabla U(x^{\ast})|^2\gamma^2 +\frac{\gamma^4}{4}|\nabla^2U(\tilde{x}_{n\gamma})\nabla U(\tilde{x}_{n\gamma})-\vec{\Delta}(\nabla U)|^2\nonumber\\
	 &\quad +\gamma^3\nabla U \vec{\Delta}(\nabla U) +4\gamma^3 d L_1^2 +4d\gamma \nonumber\\
	& \leq \left(\frac{5\gamma^2}{4} +\frac{\gamma^4L_1^2}{2}\right)|\nabla U(\tilde{x}_{n\gamma}) - \nabla U(x^{\ast})|^2 +\frac{\gamma^2 L_2^2}{2}d^2+\gamma^4 L_2^2 d^2 +4\gamma^3L_1^2 d +4\gamma d \\
	& \leq  \gamma(c_1 |\tilde{x}_{n\gamma}-x^{\ast}|^2+c_2) \nonumber, 
\end{align}
where $c_1= \frac{5L_1^2}{4} +\frac{L_1^4}{2}$ and $c_2 = \frac{3 L_2^2}{2}d^2 +4L_1^2 d +4 d$. Then, denote by 
\[
\bar{\Delta}_n = - \nabla U(\tilde{x}_{n\gamma})(t-n\gamma)+\frac{(t-n\gamma)^2}{2}\left(\nabla^2U(\tilde{x}_{n\gamma})\nabla U(\tilde{x}_{n\gamma})-\vec{\Delta}(\nabla U)(\tilde{x}_{n\gamma})\right)
\] 
and recall 
\[
\tilde{\Delta}_t = -\sqrt{2}\int_{n\gamma}^t \int_{n\gamma}^r \nabla^2U( \tilde{x}_{n\gamma})\,dw_{s}\,dr +\sqrt{2}\int_{n\gamma}^t\,dw_r.
\] 
Notice that $|\bar{\Delta}_n|^2 \leq \gamma((\frac{5L_1^2}{4} +\frac{L_1^4}{2})|\tilde{x}_{n\gamma}-x^{\ast}|^2+\frac{3 L_2^2}{2}d^2)$ by equation \eqref{auxineq}, and then one calculates
\begin{align*}
&\mathbb{E}^{\mathscr{F}_{n\gamma}}|\tilde{x}_t-\tilde{x}_{n\gamma}|^4 \\
& = \mathbb{E}^{\mathscr{F}_{n\gamma}}|\bar{\Delta}_n+\tilde{\Delta}_t |^4 \\
&= \mathbb{E}^{\mathscr{F}_{n\gamma}}\left(|\bar{\Delta}_n|^2+2\bar{\Delta}_n\tilde{\Delta}_t+|\tilde{\Delta}_t|^2\right)^2  \\
																							& \leq |\bar{\Delta}_n|^4	+6|\bar{\Delta}_n|^2\mathbb{E}^{\mathscr{F}_{n\gamma}}|\tilde{\Delta}_t|^2+\mathbb{E}^{\mathscr{F}_{n\gamma}}|\tilde{\Delta}_t|^4+4|\bar{\Delta}_n|\mathbb{E}^{\mathscr{F}_{n\gamma}}|\tilde{\Delta}_t|^3\\
																							& \leq 3 |\bar{\Delta}_n|^4	+13\mathbb{E}^{\mathscr{F}_{n\gamma}}|\tilde{\Delta}_t|^4\\
																							& \leq 9\gamma^2\left(\frac{25}{16}L_1^4 +\frac{L_1^8}{4}\right)|\tilde{x}_{n\gamma}-x^{\ast}|^4+ \gamma^2\frac{81L_2^4}{4}d^4+416\gamma^2(L_1^4 +3)d^2\\
																							& \leq  \gamma^2(c_3|\tilde{x}_{n\gamma}-x^{\ast}|^4 +c_4),
\end{align*}
where $c_3=  9 \left(\frac{25}{16}L_1^4 +\frac{L_1^8}{4}\right)$ and $c_4 = \frac{81L_2^4}{4}d^4+416(L_1^4 +3)d^2$. As for the third result, consider 
\begin{align*}
\mathbb{E}^{\mathscr{F}_{n\gamma}}\left[|x_t -x_{n\gamma} |^2\right] 	
&=\mathbb{E}^{\mathscr{F}_{n\gamma}}\left[\left|-\int_{n\gamma}^t \nabla U(x_r)\,dr +\sqrt{2}\int_{n\gamma}^t \, dw_r\right|^2\right] \\
& \leq 2\gamma  L_1^2\int_{n\gamma}^t\mathbb{E}^{\mathscr{F}_{n\gamma}}| x_r-x^{\ast}|^2\,dr  +4\gamma d\\
& \leq 2\gamma^2 L_1^2|x_{n\gamma}-x^{\ast}|^2+4\gamma^3L_1^2d+4\gamma d,
\end{align*}
where the last inequality holds by using Theorem 1 in \cite{DM16}.
\end{proof}
\begin{lemma} \label{lip4terms}
Assume \ref{h4} - \ref{h7}  are satisfied. Let  $\gamma \in \left(0,\frac{1}{\tilde{m}} \wedge  \frac{8\tilde{m}^2}{m( 2L_1^2+7\tilde{m}L_1)}\right)$. Then, for all $n \in \mathbb{N}$, and $t \in [n\gamma, (n+1)\gamma)$,
\begin{align*}
&\mathbb{E}^{\mathscr{F}_{n\gamma}}\left[|\nabla U(\tilde{x}_t)- \nabla U(\tilde{x}_{n\gamma})- \nabla U_1(t,\tilde{x}_{n\gamma})-\nabla U_2(t,\tilde{x}_{n\gamma})|^2\right]\\
& \leq \gamma^2(c_5|\tilde{x}_{n\gamma}-x^{\ast}|^4+c_6|\tilde{x}_{n\gamma}-x^{\ast}|^2+c_7),
\end{align*}
where $c_5, c_6$ and $c_7$ are given explicitly in the proof.
\end{lemma}
\begin{proof} For any $t \in [n\gamma, (n+1)\gamma)$, applying It\^o's formula to $\nabla U(\tilde{x}_t)- \nabla U(\tilde{x}_{n\gamma})$ gives, almost surely
\begin{align} \label{lipito}
\begin{split}
& \nabla U(\tilde{x}_t)- \nabla U(\tilde{x}_{n\gamma})- \nabla U_1(t,\tilde{x}_{n\gamma})-\nabla U_2(t,\tilde{x}_{n\gamma})	\\
									& = -\int_{n\gamma}^t\left(\nabla^2 U(\tilde{x}_r)-\nabla^2 U(\tilde{x}_{n\gamma})\right)\nabla U (\tilde{x}_{n\gamma})\,dr \\
									&\quad  -\int_{n\gamma}^t\nabla^2U(\tilde{x}_r)\left(\nabla U_1(r,\tilde{x}_{n\gamma})+\nabla U_2(r,\tilde{x}_{n\gamma})\right)\,dr\\
									& \quad +\sqrt{2}\int_{n\gamma}^t\left(\nabla^2 U(\tilde{x}_r)-\nabla^2 U(\tilde{x}_{n\gamma})\right)\,dw_r\\
									&\quad +\int_{n\gamma}^t\left(\vec{\Delta}(\nabla U)(\tilde{x}_r)-\vec{\Delta}(\nabla U)(\tilde{x}_{n\gamma})\right)\,dr  
\end{split}
\end{align}
Then, squaring both sides and taking conditional expectation gives
\begin{equation}\label{4termsJ}
 \mathbb{E}^{\mathscr{F}_{n\gamma}}\left[ \left| \nabla U(\tilde{x}_t)- \nabla U(\tilde{x}_{n\gamma})- \nabla U_1(t,\tilde{x}_{n\gamma})-\nabla U_2(t,\tilde{x}_{n\gamma})\right|^2\right] \leq 4\sum_{i=1}^4 \bar{G}_i(t).
\end{equation}
By using Cauchy-Schwarz inequality, \ref{h5}, \ref{h6} and Lemma \ref{liprate2}, one obtains
\begin{align*} 
\begin{split}
\bar{G}_1(t)	& \leq \gamma\int_{n\gamma}^t\mathbb{E}^{\mathscr{F}_{n\gamma}}\left[|(\nabla^2 U(\tilde{x}_r)-\nabla^2 U(\tilde{x}_{n\gamma}))\nabla U (\tilde{x}_{n\gamma})|^2\right]\,dr\\
	& \leq  \gamma L_1^2L_2^2|\tilde{x}_{n\gamma}-x^{\ast}|^2\int_{n\gamma}^t\mathbb{E}^{\mathscr{F}_{n\gamma}}|\tilde{x}_r-\tilde{x}_{n\gamma}|^2\,dr\\
	& \leq \gamma^3 (c_1L_1^2L_2^2|\tilde{x}_{n\gamma}-x^{\ast}|^4+c_2L_1^2L_2^2|\tilde{x}_{n\gamma}-x^{\ast}|^2).
\end{split}
\end{align*}
Similarly, by Cauchy-Schwarz inequality, \eqref{usefulineq} and Lemma \ref{liprate1}, we have
\begin{align*}
\begin{split}
\bar{G}_2(t)	& \leq \gamma\int_{n\gamma}^t\mathbb{E}^{\mathscr{F}_{n\gamma}}\left[|\nabla^2U(\tilde{x}_r)(\nabla U_1(r,\tilde{x}_{n\gamma})+\nabla U_2(r,\tilde{x}_{n\gamma}))|^2\right]\,dr\\
	& \leq 2\gamma L_1^2\int_{n\gamma}^t\mathbb{E}^{\mathscr{F}_{n\gamma}}\left[|\nabla U_1(r,\tilde{x}_{n\gamma})|^2+|\nabla U_2(r,\tilde{x}_{n\gamma})|^2\right]\,dr\\
	& \leq 2\gamma^2 L_1^2 (2\gamma^2(L_1^4|\tilde{x}_{n\gamma}-x^{\ast}|^2+d^2L_2^2) + 2\gamma d L_1^2) \\
	& \leq \gamma^3(4\gamma L_1^6|\tilde{x}_{n\gamma}-x^{\ast}|^2 + 4\gamma L_1^2L_2^2d^2 +4dL_1^4).
\end{split}
\end{align*}
Moreover, applying Cauchy-Schwarz inequality, \ref{h6} and Lemma \ref{liprate2} yields
\begin{align*}
\begin{split}
\bar{G}_3(t)	& = 2\int_{n\gamma}^t\mathbb{E}^{\mathscr{F}_{n\gamma}}\left[|\nabla^2 U(\tilde{x}_r)-\nabla^2 U(\tilde{x}_{n\gamma})|^2_{\mathsf{F}}\right]\,dr\\
	& \leq  2dL_2^2\int_{n\gamma}^t\mathbb{E}^{\mathscr{F}_{n\gamma}}|\tilde{x}_r-\tilde{x}_{n\gamma}|^2\,dr\\
	& \leq  \gamma^2 (2L_2^2dc_1|\tilde{x}_{n\gamma}-x^{\ast}|^2 +2L_2^2dc_2).
\end{split}
\end{align*}
Furthermore, one obtains by using Cauchy-Schwarz inequality, \ref{h7} and Lemma \ref{liprate2}
\begin{align*}
\begin{split}
\bar{G}_4(t)	& \leq \gamma\int_{n\gamma}^t\mathbb{E}^{\mathscr{F}_{n\gamma}}\left[\left|\vec{\Delta}(\nabla U)(\tilde{x}_r)-\vec{\Delta}(\nabla U)(\tilde{x}_{n\gamma})\right|^2\right]\,dr\\
	& \leq  d^{3/2}L\gamma\int_{n\gamma}^t\mathbb{E}^{\mathscr{F}_{n\gamma}}|\tilde{x}_r-\tilde{x}_{n\gamma}|^2\,dr\\
	& \leq  \gamma^3 (d^{3/2}Lc_1|\tilde{x}_{n\gamma}-x^{\ast}|^2 +d^{3/2}Lc_2).
\end{split}
\end{align*}
The proof completes by substituting all the estimates above into \eqref{4termsJ}, i.e.
\begin{align*}
&\mathbb{E}^{\mathscr{F}_{n\gamma}}\left[|\nabla U(\tilde{x}_t)- \nabla U(\tilde{x}_{n\gamma})- \nabla U_1(t,\tilde{x}_{n\gamma})-\nabla U_2(t,\tilde{x}_{n\gamma})|^2\right]\\
& \leq \gamma^2(c_5|\tilde{x}_{n\gamma}-x^{\ast}|^4+c_6|\tilde{x}_{n\gamma}-x^{\ast}|^2+c_7),
\end{align*}
where $c_5 =4c_1  L_1^2L_2^2$, $c_6 =4(L_1^2L_2^2c_2 +Lc_1d^{3/2}+2L_2^2c_1d+4L_1^6)$ and $c_7 =4(Lc_2d^{3/2} +2L_2^2c_2d+4 L_1^2L_2^2d^2+4L_1^4d)$.
\end{proof}    
\begin{lemma}\label{lipproblematic}
Assume \ref{h4} - \ref{h7} are satisfied. Let  $\gamma \in \left(0,\frac{1}{\tilde{m}} \wedge  \frac{8\tilde{m}^2}{m( 2L_1^2+7\tilde{m}L_1)}\right)$. Then, for all $n \in \mathbb{N}$, and $t \in [n\gamma, (n+1)\gamma)$,
\begin{align*}
&  \mathbb{E}^{\mathscr{F}_{n\gamma}}\left[\int_{n\gamma}^t(\nabla U(x_r)-\nabla U(\tilde{x}_r))\,dr\int_{n\gamma}^t(\nabla^2 U(\tilde{x}_r)-\nabla^2 U(\tilde{x}_{n\gamma}))\,dw_r\right]  \\
&   \leq \gamma^3( c_8|x_{n\gamma}-x^{\ast}|^2+c_9|\tilde{x}_{n\gamma}-x^{\ast}|^2+c_{10}),
\end{align*}
where the constants $c_8, c_9$ and $c_{10}$ are given explicitly in the proof.
\end{lemma}
\begin{proof}
The proof follows the same lines as in Lemma \ref{problematicterm} with $\int_{n\gamma}^tM(\tilde{x}_r, \tilde{x}_{n\gamma})\,dw_r$ replaced by $\int_{n\gamma}^t(\nabla^2 U(\tilde{x}_r)-\nabla^2 U(\tilde{x}_{n\gamma}))\,dw_r$, thus, the main focus here is to provide explicit constants. For any $t \in [n\gamma, (n+1)\gamma)$, one observes that
\begin{align}\label{problematiclip}
\begin{split}
&\mathbb{E}^{\mathscr{F}_{n\gamma}}\left[\int_{n\gamma}^t(\nabla U(x_r)-\nabla U(\tilde{x}_r))\,dr\int_{n\gamma}^t(\nabla^2 U(\tilde{x}_r)-\nabla^2 U(\tilde{x}_{n\gamma}))\,dw_r\right] \\
& = \mathbb{E}^{\mathscr{F}_{n\gamma}}\left[\int_{n\gamma}^t\left\{\nabla U(x_r)-\nabla U(x_{n\gamma})-(\nabla U(\tilde{x}_r)-\nabla U(\tilde{x}_{n\gamma}))\right. \right.\\
&\hspace{5em}\left.\left. - \sqrt{2}\int_{n\gamma}^r\nabla^2 U(x_{n\gamma}) \,dw_s + \sqrt{2}\int_{n\gamma}^r\nabla^2 U(\tilde{x}_{n\gamma}) \,dw_s\right\}\,dr \right.\\
&\hspace{15em}\left. \times \int_{n\gamma}^t(\nabla^2 U(\tilde{x}_r)-\nabla^2 U(\tilde{x}_{n\gamma}))\,dw_r\right]\\
&\quad+  \sqrt{2}\mathbb{E}^{\mathscr{F}_{n\gamma}}\left[\int_{n\gamma}^t\int_{n\gamma}^r(\nabla^2 U(x_{n\gamma})- \nabla^2U(\tilde{x}_{n\gamma})) \,dw_s\,dr\right.\\
&\hspace{15em} \left. \times \int_{n\gamma}^t(\nabla^2 U(\tilde{x}_r)-\nabla^2 U(\tilde{x}_{n\gamma}))\,dw_r\right].
\end{split}
\end{align}
By applying It\^o's formula to $\nabla^2 U(\tilde{x}_r)$, the second term in \eqref{problematiclip} can be estimated as
\begin{align}\label{lipintres}
& \sqrt{2} \mathbb{E}^{\mathscr{F}_{n\gamma}}\left[\int_{n\gamma}^t\int_{n\gamma}^r(\nabla^2 U(x_{n\gamma})- \nabla^2U(\tilde{x}_{n\gamma})) \,dw_s\,dr \right.\nonumber\\
&\hspace{15em}\left.\int_{n\gamma}^t(\nabla^2 U(\tilde{x}_r)-\nabla^2 U(\tilde{x}_{n\gamma}))\,dw_r\right] \nonumber\\
& = \sqrt{2} \mathbb{E}^{\mathscr{F}_{n\gamma}}\left[\sum_{i=1}^d \left(\int_{n\gamma}^t\sum_{l=1}^d\int_{n\gamma}^r(\nabla^2 U^{(i,l)}(x_{n\gamma})- \nabla^2U^{(i,l)}(\tilde{x}_{n\gamma})) \,dw^{(l)}_s\,dr \right)\right. \nonumber\\
&\hspace{1em} \times \left( \sum_{j=1}^d\int_{n\gamma}^t\left(-\int_{n\gamma}^r\sum_{k=1}^d \frac{\partial^3 U(\tilde{x}_s)}{\partial x^{(i)} \partial x^{(j)}\partial x^{(k)}}\nabla\tilde{U}^{(k)} (s,\tilde{x}_{n\gamma})\,ds\right. \right. \nonumber\\
&\left. \left. \left. +\int_{n\gamma}^r\sum_{k=1}^d \frac{\partial^4 U(\tilde{x}_s)}{\partial x^{(i)} \partial x^{(j)}\partial x^{(k)}\partial x^{(k)}} \,ds +\sqrt{2}\int_{n\gamma}^r\sum_{k=1}^d \frac{\partial^3 U(\tilde{x}_s)}{\partial x^{(i)} \partial x^{(j)}\partial x^{(k)}} \,dw^{(k)}_s\right)\,dw^{(j)}_r\right)\right] \nonumber\\
& \leq \frac{1}{2} \sum_{i=1}^d \mathbb{E}^{\mathscr{F}_{n\gamma}} \left| \int_{n\gamma}^t\sum_{l=1}^d\int_{n\gamma}^r(\nabla^2 U^{(i,l)}(x_{n\gamma})- \nabla^2U^{(i,l)}(\tilde{x}_{n\gamma})) \,dw^{(l)}_s\,dr \right|^2 \nonumber \\
&\hspace{1em} +\sum_{i=1}^d\mathbb{E}^{\mathscr{F}_{n\gamma}} \left|\sum_{j=1}^d\int_{n\gamma}^t \left(-\int_{n\gamma}^r\sum_{k=1}^d \frac{\partial^3 U(\tilde{x}_s)}{\partial x^{(i)} \partial x^{(j)}\partial x^{(k)}}\nabla\tilde{U}^{(k)} (s,\tilde{x}_{n\gamma})\,ds \right.\right.\nonumber\\
&\hspace{15em} \left.\left.+\int_{n\gamma}^r\sum_{k=1}^d \frac{\partial^4 U(\tilde{x}_s)}{\partial x^{(i)} \partial x^{(j)}\partial x^{(k)}\partial x^{(k)}} \,ds \right)\,dw^{(j)}_r\right|^2\nonumber\\
&\leq 2L_2^2|x_{n\gamma}-x^{\ast}|^2d\gamma^3 \nonumber\\
&\quad+ \gamma^3((2L_2^2+6L_1^2L_2^2+12 L_1^4L_2^2)d|\tilde{x}_{n\gamma}-x^{\ast}|^2+2L^2d^4 +12 L_2^4d^3+12 L_1^2L_2^2d^2).
\end{align}
where the first inequality holds due to Young's inequality and the fact that for any $i,l,j,k = 1, \dots, d$
\begin{align*}
& \mathbb{E}^{\mathscr{F}_{n\gamma}}\left[ \int_{n\gamma}^t\int_{n\gamma}^r(\nabla^2 U^{(i,l)}(x_{n\gamma})- \nabla^2U^{(i,l)}(\tilde{x}_{n\gamma})) \,dw^{(l)}_s\,dr \right.\\
 &\hspace{10em}\left. \int_{n\gamma}^t \int_{n\gamma}^r\sqrt{2} \frac{\partial^3 U(\tilde{x}_s)}{\partial x^{(i)} \partial x^{(j)}\partial x^{(k)}} \,dw^{(k)}_s\,dw^{(j)}_r\right]=0,
\end{align*}
while the last inequality holds due to Young's inequality, results in Appendix \ref{liprateMmvt} and \ref{fourthbd}, and Lemma \ref{liprate1}. By using Cauchy-Schwarz inequality, \eqref{problematiclip} becomes
\begin{align*}
& \mathbb{E}^{\mathscr{F}_{n\gamma}}\left[\int_{n\gamma}^t(\nabla U(x_r)-\nabla U(\tilde{x}_r))\,dr\int_{n\gamma}^t(\nabla^2 U(\tilde{x}_r)-\nabla^2 U(\tilde{x}_{n\gamma}))\,dw_r\right] \\
& \leq \sqrt{\mathbb{E}^{\mathscr{F}_{n\gamma}}\left|\int_{n\gamma}^t(\nabla^2 U(\tilde{x}_r)-\nabla^2 U(\tilde{x}_{n\gamma}))\,dw_r\right|^2} \left(\mathbb{E}^{\mathscr{F}_{n\gamma}}\left[\gamma \int_{n\gamma}^t|\nabla U(x_r)-\nabla U(x_{n\gamma})\right.\right.\\
&\left.\left. -(\nabla U(\tilde{x}_r)-\nabla U(\tilde{x}_{n\gamma}))- \sqrt{2}\int_{n\gamma}^r\nabla^2 U(x_{n\gamma}) \,dw_s+ \sqrt{2}\int_{n\gamma}^r\nabla^2 U(\tilde{x}_{n\gamma})\,dw_s|^2\,dr\right]\right)^{1/2}\\
& \quad +2L_2^2|x_{n\gamma}-x^{\ast}|^2d\gamma^3\\
&\quad +\gamma^3((2L_2^2+6L_1^2L_2^2+12 L_1^4L_2^2)d|\tilde{x}_{n\gamma}-x^{\ast}|^2+2L^2d^4 +12 L_2^4d^3  +12 L_1^2L_2^2d^2).
\end{align*}
Then, to estimate the first term of \eqref{problematiclip}, one applies It\^o's formula to $\nabla U(x_r)-\nabla U(x_{n\gamma})$ and $\nabla U(\tilde{x}_r)- \nabla U(\tilde{x}_{n\gamma})$ to obtain, almost surely
\begin{align*}
 & \left(\mathbb{E}^{\mathscr{F}_{n\gamma}}\left[\gamma \int_{n\gamma}^t|\nabla U(x_r)-\nabla U(x_{n\gamma})-(\nabla U(\tilde{x}_r)-\nabla U(\tilde{x}_{n\gamma}))\right.\right.\\
&\hspace{8em}\left.\left.- \sqrt{2}\int_{n\gamma}^r\nabla^2 U(x_{n\gamma}) \,dw_s+ \sqrt{2}\int_{n\gamma}^r\nabla^2 U(\tilde{x}_{n\gamma})\,dw_s|^2\,dr\right]\right)^{1/2}\\ 
& \leq 2  \left(\mathbb{E}^{\mathscr{F}_{n\gamma}}\left[\gamma^2 \int_{n\gamma}^t \int_{n\gamma}^r \left|\nabla^2U(x_s) \nabla U(x_s) - \vec{\Delta}(\nabla U)(x_s)\right|^2\,ds\,dr\right]\right.\\
& \hspace{3em} + \mathbb{E}^{\mathscr{F}_{n\gamma}}\left[\gamma^2 \int_{n\gamma}^t \int_{n\gamma}^r \left| \nabla^2 U(\tilde{x}_s) \nabla \tilde{U} (s, \tilde{x}_{n\gamma}) - \vec{\Delta}(\nabla U)(\tilde{x}_s)\right|^2\,ds\,dr\right] \\
& \hspace{3em} +2d\gamma \int_{n\gamma}^t \int_{n\gamma}^r\mathbb{E}^{\mathscr{F}_{n\gamma}}\left[ \left|\nabla^2 U(x_s)-\nabla^2 U(x_{n\gamma})\right|^2\right]\,ds\,dr \\
& \hspace{3em} \left.+2d\gamma \int_{n\gamma}^t \int_{n\gamma}^r \mathbb{E}^{\mathscr{F}_{n\gamma}}\left[ \left|\nabla^2 U(\tilde{x}_s)-\nabla^2 U(\tilde{x}_{n\gamma})\right|^2\right]\,ds\,dr\right)^{1/2}\\
& \leq 2  \left(\gamma^2 \int_{n\gamma}^t \int_{n\gamma}^r \mathbb{E}^{\mathscr{F}_{n\gamma}}\left[2L_1^4|x_s-x^{\ast}|^2+2L_2^2d^2\right]\,ds\,dr\right.\\
& \hspace{3em} + \gamma^2 \int_{n\gamma}^t \int_{n\gamma}^r \mathbb{E}^{\mathscr{F}_{n\gamma}}\left[2L_1^2\left| \nabla \tilde{U} (s, \tilde{x}_{n\gamma}) \right|^2+2L_2^2d^2\right]\,ds\,dr \\
& \hspace{3em} +2d\gamma \int_{n\gamma}^t \int_{n\gamma}^r\mathbb{E}^{\mathscr{F}_{n\gamma}}\left[ L_2^2\left|x_s -x_{n\gamma}\right|^2\right]\,ds\,dr \\
& \hspace{3em} \left. +2d\gamma \int_{n\gamma}^t \int_{n\gamma}^r \mathbb{E}^{\mathscr{F}_{n\gamma}}\left[ L_2^2\left| \tilde{x}_s - \tilde{x}_{n\gamma} \right|^2\right]\,ds\,dr\right)^{1/2}\\
& \leq 2\gamma^2 \left(2L_1^4|x_{n\gamma}-x^{\ast}|^2+4\gamma L_1^4d+2L_2^2d^2\right.\\
& \hspace{3em} + (6L_1^4+12\gamma^2L_1^6)|\tilde{x}_{n\gamma}-x^{\ast}|^2 +12\gamma^2L_1^2L_2^2d^2+12\gamma L_1^4d+2L_2^2d^2 \\
& \hspace{3em} +2d \left(2\gamma L_1^2L_2^2|x_{n\gamma}-x^{\ast}|^2  +4\gamma^2L_1^2L_2^2d+4L_2^2d\right)\\
& \hspace{3em} \left. + 2d \left( L_2^2c_1 |\tilde{x}_{n\gamma}-x^{\ast}|^2+L_2^2c_2\right)\right)^{1/2},
\end{align*}
where the first inequality holds due to Cauchy-Schwarz inequality and Young's inequality, the second inequality holds by using \eqref{usefulineq} and \ref{h6}, while the last inequality is obtained due to Lemma \ref{liprate1} and \ref{liprate2}. Finally by using Young's inequality,  
one obtains
\begin{align*}
&  \mathbb{E}^{\mathscr{F}_{n\gamma}}\left[\int_{n\gamma}^t(\nabla U(x_r)-\nabla U(\tilde{x}_r))\,dr\int_{n\gamma}^t(\nabla^2 U(\tilde{x}_r)-\nabla^2 U(\tilde{x}_{n\gamma}))\,dw_r\right]  \\
&   \leq \gamma^3( c_8|x_{n\gamma}-x^{\ast}|^2+c_9|\tilde{x}_{n\gamma}-x^{\ast}|^2+c_{10}),
\end{align*}
where $c_8 =2L_1^4+4L_1^2 L_2^2d +2L_2^2d$, $c_9 =(4L_2^2c_1+2L_2^2+6L_1^2L_2^2+12 L_1^4L_2^2)d+ 6L_1^4+12L_1^6$ and $c_{10}=2L^2d^4 +4L_2^2c_2d+12 L_2^4d^3 +32 L_1^2L_2^2d^2  +12L_2^2d^2+16 L_1^4d $.
\end{proof}
\noindent \textbf{Proof of Theorem \ref{thmwassersteinlip}.}  Note that in the Lipschitz case, there are restractions for the stepsize $\gamma \in \left(0,\frac{1}{\tilde{m}} \wedge  \frac{8\tilde{m}^2}{m( 2L_1^2+7\tilde{m}L_1)}\right)$. Consider the synchronous coupling of $x_t$ and $\tilde{x}_t$ for $t \geq 0$, where $\tilde{x}_t$ is defined by \eqref{lipscheme}. 
Let $(x_0, \tilde{x}_0)$ distributed according to $\zeta_0$, where $\zeta_0 = \pi \otimes \delta_x$ for all $x \in \mathbb{R}^d$. Define $e_t =x_t-\tilde{x}_t$, for all $t \in [n\gamma, (n+1)\gamma)$, $n \in \mathbb{N}$. By It\^o's formula, one obtains, almost surely,
\begin{align*}
|e_t|^2	& = |e_{n\gamma}|^2 -2\int_{n\gamma}^{t} e_s( \nabla U(x_s)-\nabla\tilde{ U} (s,\tilde{x}_{n\gamma}))\,ds.
\end{align*}
Then, taking the expectation and taking the derivative on both sides yield
\begin{align*}
\frac{d}{dt} \mathbb{E}\left[|e_t|^2\right] 	& = -2 \mathbb{E} \left[e_t( \nabla U(x_t)-\nabla\tilde{ U} (t,\tilde{x}_{n\gamma}))\right]\\
														& = 2\mathbb{E} \left[e_t(- (\nabla U(x_t)- \nabla U(\tilde{x}_t)))\right]\\
														& \hspace{1em} +2\mathbb{E} \left[e_t(-( \nabla U(\tilde{x}_t)- \nabla U(\tilde{x}_{n\gamma})- \nabla U_1(t,\tilde{x}_{n\gamma})-\nabla U_2(t,\tilde{x}_{n\gamma})))\right].
\end{align*}
By applying It\^o's formula to $\nabla U(\tilde{x}_t)- \nabla U(\tilde{x}_{n\gamma})$, and by calculating $\nabla U(\tilde{x}_t)- \nabla U(\tilde{x}_{n\gamma})- \nabla U_1(t,\tilde{x}_{n\gamma})-\nabla U_2(t,\tilde{x}_{n\gamma})$, one obtains \eqref{lipito}. Substituing \eqref{lipito} into the above equation and by using \ref{h4} yield
\begin{align} \label{proofthm3u1}
\frac{d}{dt} \mathbb{E} \left[|e_t|^2\right] 
														& \leq  -2m\mathbb{E} \left[|e_t|^2\right]\nonumber \\
														&  \hspace{1em} +2\mathbb{E} \left[|e_t||\int_{n\gamma}^t(\nabla^2 U(\tilde{x}_r)-\nabla^2 U(\tilde{x}_{n\gamma}))\nabla U_{\gamma}(\tilde{x}_{n\gamma})\,dr|\right] \nonumber\\
														&  \hspace{1em} +2\mathbb{E} \left[|e_t||\int_{n\gamma}^t\nabla^2U(\tilde{x}_r)(\nabla U_{1,\gamma}(r,\tilde{x}_{n\gamma})+\nabla U_{2,\gamma}(r,\tilde{x}_{n\gamma}))\,dr|\right] \nonumber\\
														&  \hspace{1em} +2\sqrt{2}\mathbb{E} \left[e_t\left(-\int_{n\gamma}^t(\nabla^2 U(\tilde{x}_r)-\nabla^2 U(\tilde{x}_{n\gamma}))\,dw_r\right)\right] \nonumber\\
														&  \hspace{1em} +2\mathbb{E} \left[|e_t||\int_{n\gamma}^t( \vec{\Delta}(\nabla U)(\tilde{x}_r)- \vec{\Delta}(\nabla U)(\tilde{x}_{n\gamma}))\,dr|\right]\\
														& \leq  -(2m-3\varepsilon)\mathbb{E} \left[|e_t|^2\right] \nonumber\\
														&  \hspace{1em} +\frac{1}{\varepsilon}\mathbb{E} \left[\left|\int_{n\gamma}^t(\nabla^2 U(\tilde{x}_r)-\nabla^2 U(\tilde{x}_{n\gamma}))\nabla U_{\gamma}(\tilde{x}_{n\gamma})\,dr\right|^2\right] \nonumber\\
														&  \hspace{1em} +\frac{1}{\varepsilon}\mathbb{E} \left[\left|\int_{n\gamma}^t\nabla^2U(\tilde{x}_r)(\nabla U_{1,\gamma}(r,\tilde{x}_{n\gamma})+\nabla U_{2,\gamma}(r,\tilde{x}_{n\gamma}))\,dr\right|^2\right] \nonumber\\
														&  \hspace{1em} +\frac{1}{\varepsilon}\mathbb{E} \left[\left|\int_{n\gamma}^t( \vec{\Delta}(\nabla U)(\tilde{x}_r)- \vec{\Delta}(\nabla U)(\tilde{x}_{n\gamma}))\,dr\right|^2\right] \nonumber\\
														&  \hspace{1em} -2\sqrt{2}\mathbb{E} \left[(e_t-e_{n\gamma})\int_{n\gamma}^t(\nabla^2 U(\tilde{x}_r)-\nabla^2 U(\tilde{x}_{n\gamma}))\,dw_r\right]  \nonumber\\ 
														& \hspace{1em} -2\sqrt{2}\mathbb{E}\left[e_{n\gamma}\int_{n\gamma}^t(\nabla^2 U(\tilde{x}_r)-\nabla^2 U(\tilde{x}_{n\gamma}))\,dw_r\right], \nonumber 
\end{align}
where the second inequality holds due to Young's inequality and the last term is zero. Then, by using the results in Lemma \ref{liprate2} and by taking $\varepsilon = \frac{m}{4}$, one obtains
\begin{align*}
&\frac{d}{dt} \mathbb{E} \left[|e_t|^2\right] \\
														& \leq  -m\mathbb{E} \left[|e_t|^2\right] \\
														&  \hspace{1em} +\frac{4}{m}\gamma^3\mathbb{E} \left[ (c_1L_1^2L_2^2|\tilde{x}_{n\gamma}-x^{\ast}|^4+c_2L_1^2L_2^2|\tilde{x}_{n\gamma}-x^{\ast}|^2) \right]\\
														&  \hspace{1em} +\frac{4}{m}\gamma^3\mathbb{E} \left[(4\gamma L_1^6|\tilde{x}_{n\gamma}-x^{\ast}|^2 + 4\gamma L_1^2L_2^2d^2 +4dL_1^4) \right]\\
														&  \hspace{1em} +\frac{4}{m} \gamma^3 \mathbb{E} \left[(d^{3/2}Lc_1|\tilde{x}_{n\gamma}-x^{\ast}|^2 +d^{3/2}Lc_2) \right]\\
														&\hspace{1em}+2\sqrt{2}\mathbb{E} \left[\int_{n\gamma}^t(\nabla U(x_r)-\nabla U(\tilde{x}_r))\,dr  \int_{n\gamma}^t(\nabla^2 U(\tilde{x}_r)-\nabla^2 U(\tilde{x}_{n\gamma}))\,dw_r \right]\\
														& \hspace{1em}+ 2\sqrt{2}\mathbb{E} \left[\int_{n\gamma}^t(\nabla U(\tilde{x}_r)- \nabla U(\tilde{x}_{n\gamma})- \nabla U_1(r,\tilde{x}_{n\gamma})-\nabla U_2(r,\tilde{x}_{n\gamma}))\,dr\right.\\
														&\hspace{8em}\left.\times \int_{n\gamma}^t(\nabla^2 U(\tilde{x}_r)-\nabla^2 U(\tilde{x}_{n\gamma}))\,dw_r \right]\\
														& \leq  -m\mathbb{E} \left[|e_t|^2\right] \\
														&  \hspace{1em} +\frac{4}{m}\gamma^3\mathbb{E} \left[ (c_1L_1^2L_2^2|\tilde{x}_{n\gamma}-x^{\ast}|^4+c_2L_1^2L_2^2|\tilde{x}_{n\gamma}-x^{\ast}|^2)\right]\\
														&  \hspace{1em} +\frac{4}{m}\gamma^3\mathbb{E} \left[(4\gamma L_1^6|\tilde{x}_{n\gamma}-x^{\ast}|^2 + 4\gamma L_1^2L_2^2d +4dL_1^4)\right]\\
														&  \hspace{1em} +\frac{4}{m} \gamma^3 \mathbb{E} \left[(d^{3/2}Lc_1|\tilde{x}_{n\gamma}-x^{\ast}|^2 +d^{3/2}Lc_2) \right]\\
														&\hspace{1em}+2\sqrt{2}\gamma^3\mathbb{E} \left[( c_8|x_{n\gamma}-x^{\ast}|^2+c_9|\tilde{x}_{n\gamma}-x^{\ast}|^2+c_{10}) \right]\\
														& \hspace{1em}+ 2 \gamma^3\mathbb{E} \left[(c_5|\tilde{x}_{n\gamma}-x^{\ast}|^4+c_6|\tilde{x}_{n\gamma}-x^{\ast}|^2+c_7) \right]\\
														&\hspace{1em} +2\gamma^3\mathbb{E} \left[d( L_2^2c_1|\tilde{x}_{n\gamma}-x^{\ast}|^2+ L_2^2c_2)\right]
\end{align*} 
where the last inequality holds by using Cauchy-Schwarz inequality, Young's inequality and Lemma \ref{liprate2}, \ref{lipproblematic}, \ref{lip4terms}. Then, after simplification, one obtains
\begin{align*}
\frac{d}{dt} \mathbb{E} \left[|e_t|^2\right] 	& \leq -m\mathbb{E} \left[|e_t|^2\right]\\
&\quad+ \gamma^3\mathbb{E} \left[(c_{11}|\tilde{x}_{n\gamma}-x^{\ast}|^4+c_{12}|\tilde{x}_{n\gamma}-x^{\ast}|^2+c_{13}|x_{n\gamma}-x^{\ast}|^2+c_{14})\right],
\end{align*}
where $c_{11} =\frac{4}{m}L_1^2L_2^2c_1+2c_5$, $c_{12} = \frac{4}{m}(L_1^2L_2^2c_2 +4L_1^6+d^{3/2}Lc_1)+2\sqrt{2}c_9+2c_6+2dL_2^2c_1$, $c_{13} = 2\sqrt{2}c_8$ and $c_{14} = \frac{4}{m}(4L_1^2L_2^2d+4L_1^4d+d^{3/2}Lc_2)+2\sqrt{2}c_{10}+2c_7+ 2L_2^2dc_2$.
Then, the application of Gronwall's lemma yields
\begin{align*}
 \mathbb{E} \left[|e_t|^2\right] 	 &\leq e^{-m(t-n\gamma)}\mathbb{E} \left[|e_{n\gamma}|^2\right]\\
 &\quad+\gamma^4\mathbb{E} \left[(c_{11}|\tilde{x}_{n\gamma}-x^{\ast}|^4+c_{12}|\tilde{x}_{n\gamma}-x^{\ast}|^2+c_{13}|x_{n\gamma}-x^{\ast}|^2+c_{14})\right].
\end{align*}
Finally, by induction, Proposition \ref{lip2ndmoment} and \ref{lip4thmoment}, one obtains
\begin{align*}
& \mathbb{E}\left[|e_{(n+1)\gamma}|^2\right]	\\
 													& \leq e^{-m\gamma (n+1)}\mathbb{E}\left[|e_0|^2\right]+\frac{\gamma^4c_{14}}{1-e^{-m\gamma}}+\gamma^4c_{11} \sum_{k=0}^n\mathbb{E}\left[|\tilde{x}_{k\gamma}-x^{\ast}|^4\right]e^{-m\gamma (n-k)}\\
 													&\hspace{1em}+\gamma^4c_{12} \sum_{k=0}^n\mathbb{E}\left[|\tilde{x}_{k\gamma}-x^{\ast}|^2\right]e^{-m\gamma (n-k)}+\gamma^4c_{13} \sum_{k=0}^n\mathbb{E}\left[|x_{k\gamma}-x^{\ast}|^2\right]e^{-m\gamma (n-k)}\\
 													& \leq e^{-m\gamma (n+1)}\mathbb{E}\left[|x_0-\tilde{x}_0|^2\right]+\frac{\gamma^3 e^{m }}{m}\left(c_{14}+c_{11}\left(\mathbb{E}\left[|\tilde{x}_0-x^{\ast}|^4\right] +\frac{8q_2}{\tilde{m}}\right)\right. \\
 													&\quad \left. +c_{12}\left(\mathbb{E}\left[|\tilde{x}_0-x^{\ast}|^2\right] +\frac{q_1}{\tilde{m}}\right)+c_{13}\left(\mathbb{E}\left[|x_0-x^{\ast}|^2\right] +2d\right) \right)
\end{align*}
where the last inequality holds by using $1-e^{-m \gamma } \geq m \gamma e^{-m \gamma}$. The application of Theorem 1 in \cite{DM16} with the initial distribution $\zeta_0$ yields
\[
W_2^2(\delta_x \tilde{R}_{\gamma}^n, \pi) \leq e^{-mn\gamma}\left(2|x-x^{\ast}|^2+\frac{2d}{m}\right)+\bar{C}\gamma^3,
\]
where $\bar{C} =O(d^4)$ with
\[
\bar{C} = \frac{  e^{m }}{m}\left(c_{14}+c_{11}\left(|x-x^{\ast}|^4+\frac{8q_2}{\tilde{m}}\right)+c_{12}\left(|x-x^{\ast}|^2+\frac{q_1}{\tilde{m}}\right)+c_{13}\left(\frac{d}{m} +2d\right) \right)
\]

\noindent \textbf{Proof of Corollary \ref{thmwassersteingau}.} In the case that the target distribution $\pi$ is a multivariate Gaussian distribution, by using the same arguments, one notices that for $\gamma \in \left(0,\frac{1}{\tilde{m}} \wedge  \frac{8\tilde{m}^2}{m( 2L_1^2+7\tilde{m}L_1)}\right)$, Proposition \ref{lip2ndmoment} holds with $q_1 = (4L_1^2 +4)d$. Then, one obtains the following bound
\begin{align*}
 \mathbb{E} \left[|e_t|^2\right] 	 \leq e^{-m(t-n\gamma)}\mathbb{E} \left[|e_{n\gamma}|^2\right]+\gamma^4\mathbb{E} \left[\frac{4}{m}\left(4L_1^6 |\tilde{x}_{n\gamma}-x^{\ast}|^2 +4dL_1^4\right)\right],
\end{align*}
which indicates
\[
W_2^2(\delta_x \tilde{R}_{\gamma}^n, \pi) \leq e^{-mn\gamma}\left(2|x-x^{\ast}|^2+\frac{2d}{m}\right)+\tilde{C}\gamma^3,
\]
where $\tilde{C}= \frac{16L_1^4e^m}{m^2}\left(d+ L_1^2\left(|x-x^{\ast}|^2+\frac{q_1}{\tilde{m}}\right)\right) $.

\subsection{Example: Logistic regression with Gaussian prior} \label{exlip}
We provide an example of the logistic regression in dimension $d$. Denote by $\theta_k$, $k \in \mathbb{N}$ the k-th iteration of the algorithm \eqref{eqnschemelip}. One observes a sequence of i.i.d. sample $\{(x_i, y_i)\}_{i = 1, \dots, n}$, where $x_i \in \mathbb{R}^d$ and $y_i \in \{0,1\}$ for all $i$. The likelihood function is given by $p(y_i |x_i, \theta) = (1/(1+e^{-x_i^{\mathsf{T}}\theta}))^{y_i}(1-1/(1+e^{-x_i^{\mathsf{T}}\theta}))^{1-y_i}$. Consider a Gaussian prior with mean zero and covariance matrix proportional to the inverse of the matrix $\Sigma_X = \frac{1}{n}\sum_{i = 1}^nx_ix_i^{\mathsf{T}}$. For $\theta \in \mathbb{R}^d$, the gradient $\nabla U(\theta)$ and Hessian $\nabla^2 U(\theta)$ with $n$ data points are
\[
\nabla U(\theta) =c \Sigma_X \theta +  \sum_{i = 1}^n\left( \frac{x_i}{1+e^{-x_i^{\mathsf{T}}\theta}} - y_ix_i\right), \quad
\nabla^2 U(\theta) =c\Sigma_X + \sum_{i = 1}^n \frac{x_ix_i^{\mathsf{T}}e^{-x_i^{\mathsf{T}} \theta}}{(1+e^{-x_i^{\mathsf{T}} \theta})^2},
\]
where $c>0$. This implies that $L_1 \leq (c+n)\max_i|x_i x_i^{\mathsf{T}}|$ with $|x_i x_i^{\mathsf{T}}|$ the spectral norm of the matrix $x_i x_i^{\mathsf{T}}$ for each $i$. One notices that $\max_i|x_i x_i^{\mathsf{T}}|$ is much smaller than $\max_i|x_i|^2 = O(d)$ due to the fact that the matrix $x_i x_i^{\mathsf{T}}$ is typically sparse in statistical and machine learning applications. One may refer to dimension reduction techniques in sparse matrices in data science for more discussions, see e.g. \cite{sparsedr1} and \cite{sparsedr2}. 

To calculate the Lipschitz constant $L_2$ in \ref{h6}, one denotes $g(\lambda) = \nabla^2 U(\lambda y + (1-\lambda)x)$, for any $x, y \in \mathbb{R}^d$ and $\lambda \in [0,1]$. By fundamental theorem of calculus, one obtains, for any $l = 1, \dots, d$
		\begin{align*}
		g^{(l,\cdot)}(1) - g^{(l,\cdot)}(0) 	&=  \nabla^2 U^{(l,\cdot)}(y) -  \nabla^2 U^{(l,\cdot)}(x) \\
												&= \int_0^1 \nabla^2 ( \nabla U)^{(l)}(\lambda y+(1-\lambda )x)(y-x) \, d\lambda,
		\end{align*}
		where $ \nabla^2 ( \nabla U)^{(l)}$ is a matrix with $(j,k)$-th entry $\frac{\partial^3 U}{\partial x^l \partial x^j \partial x^k}$ and for any $\theta \in \mathbb{R}^d$
		\[
		  \nabla^2 ( \nabla U)^{(l)}(\theta)  = \sum_{i = 1}^n \left( \frac{2x_i^{(l)}x_i x_i^{\mathsf{T}}e^{-2x_i^{\mathsf{T}} \theta}}{(1+e^{-x_i^{\mathsf{T}} \theta})^3}- \frac{x_i^{(l)}x_i x_i^{\mathsf{T}}e^{-x_i^{\mathsf{T}} \theta}}{(1+e^{-x_i^{\mathsf{T}} \theta})^2} \right) 
		\]
		Moreover, one notices
		\begin{align*}
		|\nabla^2 U(y) - \nabla^2U(x)| & \leq \left(\sum_{l = 1}^d \left|\nabla^2 U^{(l,\cdot)}(y) -  \nabla^2 U^{(l,\cdot)}(x) \right|^2\right)^{1/2}\\
												& \leq 3n\max_i|x_i||x_i x_i^{\mathsf{T}}||y-x| 
		\end{align*}
which implies $L_2 = 3n\max_i|x_i||x_i x_i^{\mathsf{T}}|$. 

Finally, for the constant $L$ in \ref{h7}, define for any $k = 1, \dots, d$, $f_k(\lambda) = \nabla^2(\nabla U)^{(k)}(\lambda y + (1-\lambda)x)$, for any $x, y \in \mathbb{R}^d$ and $\lambda \in [0,1]$, and one uses the same technique to obtain, for any $l= 1,\dots, d$
		\begin{align*}
		f_k^{(l,\cdot)}(1) - f_k^{(l,\cdot)}(0) 	&= (\nabla^2(\nabla U)^{(k)})^{(l,\cdot)}(y) -  (\nabla^2(\nabla U)^{(k)})^{(l,\cdot)}(x)\\
											& = \int_0^1 \nabla^2(\nabla^2 U)^{(k,l)}(\lambda y+(1-\lambda )x)(y-x) \, d\lambda.
		\end{align*}
		where $\nabla^2 ( \nabla^2 U)^{(k,l)}$ is a matrix with $(j,m)$-th entry $\frac{\partial^4 U}{\partial x^k \partial x^l\partial x^j \partial x^m}$ and for any $\theta \in \mathbb{R}^d$
		\begin{align*}
		\nabla^2(\nabla^2 U)^{(k,l)} (\theta)  & \leq \sum_{i = 1}^n \left| \frac{x_i^{(k)}x_i^{(l)}x_ix_i^{\mathsf{T}}e^{-x_i^{\mathsf{T}} \theta}}{(1+e^{-x_i^{\mathsf{T}} \theta})^2} \right.\\
				&\hspace{6em}\left. - \frac{6x_i^{(k)}x_i^{(l)}x_ix_i^{\mathsf{T}}e^{-2x_i^{\mathsf{T}} \theta}}{(1+e^{-x_i^{\mathsf{T}} \theta})^3} + \frac{6x_i^{(k)}x_i^{(l)}x_ix_i^{\mathsf{T}}e^{-3x_i^{\mathsf{T}} \theta}}{(1+e^{-x_i^{\mathsf{T}} \theta})^4}\right| 
		\end{align*}
		Then, one obtains for $k = 1, \dots, d$,
		\begin{align*}
		&|\nabla^2(\nabla U)^{(k)}(y) - \nabla^2(\nabla U)^{(k)}(x)| \\
											& =  \left(\sum_{l = 1}^d  \left|(\nabla^2(\nabla U)^{(k)})^{(l,\cdot)}(y) -  (\nabla^2(\nabla U)^{(k)})^{(l,\cdot)}(x)\right|^2\right)^{1/2}\\
											& \leq 13n\max_i|x_i^{(k)}||x_i||x_i x_i^{\mathsf{T}}| |y-x|,
		\end{align*}		
		which implies $L  \leq   13n\max_i|x_i^{(k)}||x_i||x_i x_i^{\mathsf{T}}|$, and effectively, it has the same dimension dependence as $L_2$.

\appendix

\section{Proof of Remark \ref{poly}}\label{remark2}
\ref{h3} states there exists $L>0$, $\rho \geq 2$, and $\beta \in (0,1]$, such that for any $i = 1, \dots, d$ and for all $x,y \in \mathbb{R}^d$,
\[
|\nabla^2(\nabla U)^{(i)}(x)-\nabla^2(\nabla U)^{(i)}(y)|\leq L(1+|x|+|y|)^{\rho -2}|x-y|^{\beta}.
\]
By {\bf H2}, one obtains
\[
|\nabla^2(\nabla U)^{(i)}(x)| \leq  L(1+|x|)^{\rho -2}|x|^{\beta} +|\nabla^2(\nabla U)^{(i)}(0)| \leq K(1+|x|)^{\rho -2+\beta},
\]
where $K = \max\{L, |\nabla^2(\nabla U)^{(i)}(0)|\}$. Then by fundamental theorem of calculus,
\begin{align*}
|\nabla(\nabla U)^{(i)}(x)-\nabla(\nabla U)^{(i)}(y)|	& =\left|\int_0^1 \nabla^2(\nabla U)^{(i)}(tx +(1-t)y)\,dt (x-y)\right| \\
																& \leq \int_0^1 |\nabla^2(\nabla U)^{(i)}(tx +(1-t)y)|\,dt |x-y|\\
																& \leq \int_0^1 K(1+|x|+|y|)^{\rho -2+\beta}\, dt |x-y|\\
																&\leq K(1+|x|+|y|)^{\rho -2+\beta}|x-y|.
\end{align*}
Moreover, notice that
\begin{align*}
|\nabla^2 U(x)-\nabla^2 U(y)|		&\leq |\nabla^2 U(x)-\nabla^2 U(y)|_{\mathsf{F}}\\
 										& = \left(\sum_{i = 1}^d \sum_{j = 1}^d \left|\frac{\partial^2 U(x)}{\partial x^{(i)}\partial x^{(j)}}  - \frac{\partial^2 U(y)}{\partial x^{(i)}\partial x^{(j)}} \right|^2\right)^{1/2}\\
 										& = \left(\sum_{i = 1}^d |\nabla(\nabla U)^{(i)}(x)-\nabla(\nabla U)^{(i)}(y)|^2 \right)^{1/2}\\
 										& \leq \sqrt{d}K(1+|x|+|y|)^{\rho -2+\beta}|x-y|.
\end{align*}
Furthermore, one obtains
\begin{align*}
&|\vec{\Delta}(\nabla U)(x) -\vec{\Delta}(\nabla U)(y)|	\\
& =  \left(\sum_{i = 1}^d  \left|\sum_{j = 1}^d \frac{\partial^3 U(x)}{\partial x^{(i)}\partial x^{(j)}\partial x^{(j)}}  - \frac{\partial^3 U(y)}{\partial x^{(i)}\partial x^{(j)}\partial x^{(j)}} \right|^2\right)^{1/2}\\
																& \leq \left(d\sum_{i = 1}^d \sum_{j = 1}^d \left| \frac{\partial^3 U(x)}{\partial x^{(i)}\partial x^{(j)}\partial x^{(j)}}  - \frac{\partial^3 U(y)}{\partial x^{(i)}\partial x^{(j)}\partial x^{(j)}} \right|^2\right)^{1/2}\\
																& \leq \left(d\sum_{i = 1}^d |\nabla^2(\nabla U)^{(i)}(x)-\nabla^2(\nabla U)^{(i)}(y)|_{\mathsf{F}}^2\right)^{1/2}\\
																&  \leq \left(d^2\sum_{i = 1}^d L^2(1+|x|+|y|)^{2\rho -4}|x-y|^{2\beta}\right)^{1/2}\\
																& =  d^{3/2}L(1+|x|+|y|)^{\rho-2}|x-y|^{\beta}.
\end{align*}
Notice that the last inequality in Remark \ref{poly} is not obtained directly by using the above result, but it is obtained by using the arguments in page 24 of \cite{dal17user}. However, the rest of the inequalities in Remark \ref{poly} can be obtained by using similar arguments as above. 

\section{Proof of inequality \eqref{logsob} in Proposition \ref{momentbound} }\label{logsobproof}
In order to prove \eqref{logsob}, one needs the following definition and the propositions. 
\begin{definition} Consider a probability measure space $(\mathbb{R}^d,\mathcal{B}(\mathbb{R}^d), \nu)$. Let $\mathcal{A}$ be the set of continuously differentiable, Lipschitz functions on $\mathbb{R}^d$. We say that $\nu$ satisfies a Log-Sobolev inequality if there exists $C>0$ such that
\[
\mathrm{Ent}_{\nu}(f^2) \leq 2 C\int_{\mathbb{R}^d}|\nabla f|^2 \,d\nu,
\]
for every function $f \in \mathcal{A}$ with $\mathrm{Ent}_{\nu}(f^2log^{+}f^2)< \infty$, where
\[
\mathrm{Ent}_{\nu}(f) = \mathbb{E}_{\nu}(f\log f) - \mathbb{E}_{\nu}(f)\log \mathbb{E}_{\nu}(f).
\]
\end{definition}
For more details about the definition of the Log-Sobolev inequality, please refer to Chapter 2 in \cite{ledouxappendix}.
\begin{prop}[Proposition 5.4.1 in \cite{ledoux}] \label{expint} If $\nu$ satisfies a logarithmic Sobolev inequality with constant $C >0$, then for every 1-Lipschitz function $f$ and every $\alpha^2 <1/C$,
\[
\int_{\mathbb{R}^d} e^{\alpha^2 f^2/2} \,d\nu < \infty.
\]
More precisely, any 1-Lipschitz function $f$ is integrable and for every $s \in \mathbb{R}$,
\[
\int_{\mathbb{R}^d} e^{sf} \,d\nu < e^{s \int_{\mathbb{R}^d} f\,d\nu + Cs^2/2}.
\]
\end{prop}
\begin{prop} [Proposition 5.5.1 in \cite{ledoux}] \label{lsstdgaus}The standard Gaussian measure $\nu$ on the Borel sets of $\mathbb{R}^d$ satisfies, for every $f \in \mathcal{A}$,
\[
\mathrm{Ent}_{\nu}(f^2) \leq 2 \int_{\mathbb{R}^d}|\nabla f|^2 \,d\nu.
\]
\end{prop}
Proposition \ref{lsstdgaus} implies that, for a Gaussian measure $\nu$ with mean $\mu$ and covariance matrix $Q$, by using change of variables, one obtains for every $f \in \mathcal{A}$ on $\mathbb{R}^d$,
\begin{equation}\label{lsgaus}
\mathrm{Ent}_{\nu}(f^2) \leq 2 \int_{\mathbb{R}^d}( Q \nabla f ) \nabla f \,d\nu.
\end{equation}
One notes that the scheme \eqref{eqnscheme2} shows that for any $n \in \mathbb{N}$ and $x \in \mathbb{R}^d$, conditional on the previous step $\overline{X}_{n-1}=x$, $\overline{X}_{n}$ is a Gaussian random variable with mean $\mu (x) = x+ \mu_{\gamma}(x)\gamma$ where 
\[
\mu_{\gamma}(x) = -\nabla U_{\gamma}( x) +\frac{\gamma}{2}\left(\left(\nabla^2 U\nabla U\right)_{\gamma}( x)-\vec{\Delta}(\nabla U)_{\gamma}(x)\right),
\] 
and covariance matrix $Q(x) =2\gamma\left(\mathbf{I}_d - \gamma \nabla^2 U_{\gamma}(x)+\frac{\gamma^2}{3}(\nabla^2 U_{\gamma}(x))^2\right)$. 
Then, by using \eqref{lsgaus}, one obtains
\[
\mathrm{Ent}_{\nu}(f^2) \leq 2 \int_{\mathbb{R}^d} (Q \nabla f )\nabla f \,d\nu  \leq 2 \int_{\mathbb{R}^d}\frac{14}{3}\gamma |\nabla f|^2 \,d\nu.
\]
Therefore, applying Proposition \ref{expint} with $s=a$, $f = \sqrt{1+|x|^2}$ and $C =\frac{14}{3}\gamma$ yields the desired result, i.e.
\[
R_{\gamma}V_a(x) =\mathbb{ E}_x(V_a(\overline{X}_1))\leq e^{\frac{7}{3}\gamma a^2}\exp\left\{a\mathbb{ E}((1+|\overline{X}_1|^2)^{1/2}|\overline{X}_0=x)\right\}.
\]

\section{Proof of inequality \eqref{wassersteinindeq} in Theorem \ref{thmwasserstein}}\label{wassersteinind}
To obtain \eqref{wassersteinindeq}, one consider the following cases
\begin{enumerate}[label=(\roman*)]
\item If $m >\frac{7}{3}c^2$,
\begin{align*}
\begin{split}
&C\gamma^{3+\beta}\mathbb{E}\left[V_c(\bar{x}_0)\right]\sum_{k=0}^ne^{-\frac{7}{3}c^2\gamma k-m\gamma (n-k)} \\
& = C\gamma^{3+\beta}e^{-m\gamma n}\mathbb{E}\left[V_c(\bar{x}_0)\right]\sum_{k=0}^ne^{-\frac{7}{3}c^2\gamma k+m\gamma k}\\
	&= C\gamma^{3+\beta}e^{-m\gamma n}\mathbb{E}\left[V_c(\bar{x}_0)\right]\frac{e^{(n+1)(m-\frac{7}{3}c^2)\gamma}-1}{e^{(m-\frac{7}{3}c^2)\gamma}-1}\\
	& \leq C\gamma^{3+\beta}e^{-m\gamma n}\mathbb{E}\left[V_c(\bar{x}_0)\right]\frac{e^{n(m-\frac{7}{3}c^2)\gamma}}{1-e^{-(m-\frac{7}{3}c^2)\gamma}}\\
	&\leq \frac{C\mathbb{E}\left[V_c(\bar{x}_0)\right]}{m-\frac{7}{3}c^2}e^{m\gamma}\gamma^{2+\beta}e^{-\frac{7}{3}c^2(n+1)\gamma}.
\end{split}
\end{align*}
\item For the case $m<\frac{7}{3}c^2$, we have
\begin{align*}
\begin{split}
 &C\gamma^{3+\beta}e^{-m\gamma n}\mathbb{E}\left[V_c(\bar{x}_0)\right]\sum_{k=0}^ne^{-\frac{7}{3}c^2\gamma k+m\gamma k}\\
 &\leq C\gamma^{3+\beta}e^{-m\gamma n}\mathbb{E}\left[V_c(\bar{x}_0)\right]\frac{1}{1-e^{-(\frac{7}{3}c^2-m)\gamma}}\\
  	&\leq \frac{C\mathbb{E}\left[V_c(\bar{x}_0)\right]}{\frac{7}{3}c^2-m}e^{\frac{7}{3}c^2\gamma}\gamma^{2+\beta}e^{-m(n+1)\gamma}.
\end{split}
\end{align*}
\item As for the case $m=\frac{7}{3}c^2$, it can be shown that
\begin{align*}
&C\gamma^{3+\beta}e^{-m\gamma n}\mathbb{E}\left[V_c(\bar{x}_0)\right]\sum_{k=0}^ne^{-\frac{7}{3}c^2\gamma k+m\gamma k}	\\
& =C(n+1)\gamma^{3+\beta}e^{-m\gamma n}\mathbb{E}\left[V_c(\bar{x}_0)\right]\\
 &\leq \frac{C\mathbb{E}\left[V_c(\bar{x}_0)\right]}{m}e^{m\gamma}\gamma^{2+\beta}.
\end{align*}
\end{enumerate}

\section{Proof of inequality \eqref{lipintres} in Lemma \ref{lipproblematic}}\label{liprateMmvt}
For all $x, y \in \mathbb{R}^d$ and a constant $c>0$, denote by $g(t) = \nabla^2 U(x+tc(y-x))$. One notes that for any $i,j = 1, \dots,d$, $(g^{(i,j)})'(t) = c\sum_{k=1}^d \frac{\partial^3 U(x+tc(y-x))}{\partial x^{(i)} \partial x^{(j)}\partial x^{(k)}}(y^{(k)}-x^{(k)})$. By mean value theorem, there exists $t_{ij} \in [0,1]$, such that
\[
 \nabla^2 U^{(i,j)}(x+c(y-x))- \nabla^2 U^{(i,j)}(x) = g^{(i,j)}(1)-g^{(i,j)}(0)=(g^{(i,j)})'(t_{ij}).
\]
Then, one obtains
\begin{align*}
&|\nabla^2 U(x+c(y-x))- \nabla^2 U(x)|_{\mathsf{F}}\\
&=|g(1)- g(0)|_{\mathsf{F}} \\
									& =c \sqrt{\sum_{i,j=1}^d\left| \sum_{k=1}^d \frac{\partial^3 U(x+t_{ij}c(y-x))}{\partial x^{(i)} \partial x^{(j)}\partial x^{(k)}}(y^{(k)}-x^{(k)})\right|^2}\\
									& \leq \sqrt{d} L_2 |c(y-x)|,
\end{align*}
which, by sending $c$ to zero yields
\[
\sqrt{\sum_{i,j=1}^d\left| \sum_{k=1}^d \frac{\partial^3 U(x)}{\partial x^{(i)} \partial x^{(j)}\partial x^{(k)}}(y^{(k)}-x^{(k)})\right|^2}   \leq  \sqrt{d} L_2 |y-x|.
\]

\section{Proof of inequality \eqref{lipintres} in Lemma \ref{lipproblematic}} \label{fourthbd}
For any $x \in \mathbb{R}^d$, our goal is to find an upper bound for
\[
\sum_{i,j = 1}^d \left|\sum_{k = 1}^d\frac{\partial^4 U(x)}{\partial x^{(i)} \partial x^{(j)}\partial x^{(k)}\partial x^{(k)}} \right|^2 \leq d \sum_{k = 1}^d \sum_{i,j = 1}^d \left|\frac{\partial^4 U(x)}{\partial x^{(i)} \partial x^{(j)}\partial x^{(k)}\partial x^{(k)}} \right|^2.
\]
For any $i, j, k = 1, \dots, d$, for all $x, y \in \mathbb{R}^d$ and a constant $c>0$, define a function $g: \mathbb{R} \rightarrow \mathbb{R}^d$ by
\[
g^{(k)}_{(i,j)}(t) = (\nabla(\nabla^2 U)^{(i,j)}(x+tcy))^{(k)} =  \frac{\partial^3 U(x+tcy)}{\partial x^{(i)} \partial x^{(j)}\partial x^{(k)}}.
\]
One notes that by mean value theorem, there exists $t_k \in [0,1]$, such that
\begin{align*}
g^{(k)}_{(i,j)}(1) - g^{(k)}_{(i,j)}(0) 	&= ( \nabla(\nabla^2 U)^{(i,j)} (x+ cy))^{(k)}- ( \nabla(\nabla^2 U)^{(i,j)}(x))^{(k)}\\
											&= c \sum_{l = 1}^d(\nabla^2( \nabla^2 U)^{(i,j)}(x+t_k cy) )^{(k,l)}y^{(l)}.
\end{align*} 
Then, since
\begin{align*}
&\left| \nabla((\nabla^2 U)^{(i,j)}(x+ cy)) -  \nabla((\nabla^2 U)^{(i,j)}(x)) \right|  \\
&  = \left(\sum_{k = 1}^d \left| ( \nabla(\nabla^2 U)^{(i,j)} (x+ cy))^{(k)}- ( \nabla(\nabla^2 U)^{(i,j)}(x))^{(k)}\right|^2\right)^{1/2} \\
& =c\left(\sum_{k = 1}^d \left|  \sum_{l = 1}^d(\nabla^2( \nabla^2 U)^{(i,j)}(x+t_k cy) )^{(k,l)}y^{(l)}\right|^2\right)^{1/2} \\
& = \left(\sum_{k = 1}^d \left| \frac{\partial^3 U(x+cy)}{\partial x^{(i)} \partial x^{(j)}\partial x^{(k)}} -  \frac{\partial^3 U(x)}{\partial x^{(i)} \partial x^{(j)}\partial x^{(k)}}\right|^2\right)^{1/2}\\ 
& \leq \left|\nabla^2(\nabla U)^{(i)}(x+cy)-\nabla^2(\nabla U)^{(i)}(x)\right|_{\mathsf{F}}\\
& \leq \sqrt{d}Lc|y|,
\end{align*}
one obtains for any $i, j = 1, \dots, d$ and $x \in \mathbb{R}^d$,
\[
\left(\sum_{k = 1}^d \left|  \sum_{l = 1}^d(\nabla^2( \nabla^2 U)^{(i,j)}(x+t_k cy) )^{(k,l)}y^{(l)}\right|^2\right)^{1/2} \leq \sqrt{d}L|y|,
\]
which, by sending $c$ to zero yields 
\[
\left| \nabla^2( \nabla^2 U)^{(i,j)}(x) y\right| \leq \sqrt{d}L|y|
\]
and this implies $\left|\nabla^2( \nabla^2 U)^{(i,j)}(x) \right| \leq \sqrt{d}L$. Finally, we have for any $x \in \mathbb{R}^d$,
\begin{align*}
 d \sum_{k = 1}^d \sum_{i,j = 1}^d \left|\frac{\partial^4 U(x)}{\partial x^{(i)} \partial x^{(j)}\partial x^{(k)}\partial x^{(k)}} \right|^2 &\leq d \sum_{k = 1}^d \left|\nabla^2( \nabla^2 U)^{(k,k)}(x) \right|^2 _{\mathsf{F}}\\
& \leq d^2\sum_{k = 1}^d\left|\nabla^2( \nabla^2 U)^{(k,k)}(x) \right|^2 \leq d^4L^2.
\end{align*}


\end{document}